\documentclass[12pt]{article}
\usepackage{amsfonts}
\usepackage{amsthm}
\usepackage{url}
\usepackage{mathtools}
\usepackage[all]{xy}
\usepackage{amssymb}
\usepackage{color}
\usepackage{mathrsfs}
\usepackage{tikz}
\usepackage{hyperref}
\usetikzlibrary{calc,intersections,through,backgrounds}

\usepackage{pdfsync}
\newtheorem{thm}{Theorem}[section]
\newtheorem{prop}[thm]{Proposition}
\newtheorem{lemma}[thm]{Lemma}
\newtheorem{coro}[thm]{Corollary}

\newtheorem{rem}{Remark}[thm]

\newcommand{\Z}{{\mathbb Z}}

\newcommand{\Sp}{{\mathrm{Spec} \:}}

\newcommand{\cI}{{\mathrm{c\text{--} Ind}}}
\DeclareMathOperator{\Spm}{\mathrm{m-Spec}}

\makeatletter
\let\c@equation\c@thm
\makeatother
\numberwithin{equation}{section}

\title{The endomorphism ring of projectives and the Bernstein centre}

\author{Alexandre Pyvovarov}

\date{\today}

\begin{document}

\maketitle

\begin{abstract}

Let $F$ be a local non-archimedean field and $\mathcal{O}_F$ its ring of integers. Let $\Omega$ be a Bernstein component of the category of smooth representations of $GL_n(F)$, let $(J, \lambda)$ be a Bushnell-Kutzko $\Omega$-type, and let $\mathfrak{Z}_{\Omega}$ be the centre of the Bernstein component $\Omega$. This paper contains two major results. Let $\sigma$ be a direct summand of $\mathrm{Ind}_J^{GL_n(\mathcal{O}_F)} \lambda$. We will begin by computing $\mathrm{c\text{--} Ind}_{GL_n(\mathcal{O}_F)}^{GL_n(F)} \sigma\otimes_{\mathfrak{Z}_{\Omega}}\kappa(\mathfrak{m})$, where $\kappa(\mathfrak{m})$ is the residue field at maximal ideal $\mathfrak{m}$ of $\mathfrak{Z}_{\Omega}$, and the maximal ideal $\mathfrak{m}$ belongs to a Zariski-dense set in $\mathrm{Spec}\: \mathfrak{Z}_{\Omega}$. This result allows us to deduce that the endomorphism ring $\mathrm{End}_{GL_n(F)}(\mathrm{c\text{--} Ind}_{GL_n(\mathcal{O}_F)}^{GL_n(F)} \sigma)$ is isomorphic to  $\mathfrak{Z}_{\Omega}$, when $\sigma$ appears with multiplicity one in $\mathrm{Ind}_J^{GL_n(\mathcal{O}_F)} \lambda$.

\end{abstract}

\tableofcontents

\section{Introduction}\label{H}

Let $G=GL_n(F)$, where $F$ is a non-archimedean local field and let $E$ be an algebraically closed field of characteristic zero. We will consider smooth $E$-representations of $G$ and of its subgroups. Let $\Omega$ be a Bernstein component of the category of smooth representations of $G$, let $(J, \lambda)$ be a Bushnell-Kutzko $\Omega$-type, such that $J$ is contained in a maximal compact subgroup $K$. We refer the reader to notation section for the definitions.

In \cite{MR1728541} section 6 (just above Proposition 2) the authors define irreducible $K$-representations $\sigma_{\mathcal{P}}(\lambda)$, where $\mathcal{P}$ is  partition valued functions with finite support (cf. section 2 \cite{MR1728541}). One has the decomposition :

\begin{equation}\label{decomp}
\mathrm{Ind}_{J}^{K} \lambda = \bigoplus_{\mathcal{P}} \sigma_{\mathcal{P}}(\lambda)^{\oplus m_{\mathcal{P},\lambda}},
\end{equation}

\noindent where the summation runs over partition valued functions with finite support. The integers $m_{\mathcal{P},\lambda}$ are finite and we call them multiplicities of $\sigma_{\mathcal{P}}(\lambda)$'s. 

There is a natural partial ordering, as defined in \cite{MR1728541}, on the set of partition valued functions. Let $\mathcal{P}_{max}$ be the maximal partition valued function and let $\mathcal{P}_{min}$ the minimal one. Define $\sigma_{max}(\lambda):=\sigma_{\mathcal{P}_{max}}(\lambda)$ and $\sigma_{min}(\lambda):=\sigma_{\mathcal{P}_{min}}(\lambda)$. Both $\sigma_{max}(\lambda)$ and $\sigma_{min}(\lambda)$ occur in $\mathrm{Ind}_{J}^{K} \lambda$ with multiplicity 1 (see Remark \ref{mult}).

\medskip

When the type $(J, \lambda)$ is trivial, $\sigma_{min}(\lambda)$ is $st$, which is the inflation of Steinberg representation of $GL_n(k_F)$ to $K$ and $\sigma_{max}(\lambda)$ is the trivial representation. In this simplest case, when have $\Omega = [T,1]_G$, it follows from \cite[Thm 4.1]{MR1676870} that the action of $\mathfrak{Z}_{\Omega}$ on $\cI_K^G 1$ induces a ring isomorphism $\mathfrak{Z}_{\Omega} \simeq \mathrm{End}_G( \cI_K^G 1)$. The isomorphism $\mathfrak{Z}_{\Omega} \simeq \mathrm{End}_G( \cI_K^G st)$ is a special case of \cite[Corollary 6.1]{MR3508741}. It was pointed out to us by an anonymous referee that the case $\sigma = st$ can be easily deduced from the case $\sigma = 1$ via the Iwahori-Matsumuto involution(cf. \cite{MR1341660}[section 2.2]).

The purpose of this paper is to generalize the isomorphisms above to other Bernstein components and to other multiplicity one direct summands of $\mathrm{Ind}_{J}^{K} \lambda$. When the type $(J, \lambda)$ is trivial, $\sigma_{max}(\lambda)=1$ and $\sigma_{min}(\lambda)=st$  are the only multiplicity free summands of $\mathrm{Ind}_{J}^{K} \lambda$. However, when the type is semi-simple there can be more then 2 multiplicity one direct summands of $\mathrm{Ind}_{J}^{K} \lambda$ (see Remark \ref{mult}). 

In order to get such an isomorphism we will study the $\mathfrak{Z}_{\Omega}$-module \\$\mathrm{Hom}_G(\cI_K^G \sigma_{\mathcal{P}}(\lambda), \cI_K^G \sigma_{\mathcal{P}'}(\lambda))$ for any partition valued functions $\mathcal{P}$ and $\mathcal{P}'$.

Before we state precisely our main results we will introduce some more notation. Recall that $K$ is a maximal compact open subgroup of $G$ containing $J$. Let  $\widehat{K}$ denote the set of all isomorphism classes of irreducible representations of $K$. In order to simplify the notation, the decomposition (\ref{decomp}):
$$\mathrm{Ind}_{J}^{K} \lambda = \bigoplus_{\mathcal{P}} \sigma_{\mathcal{P}}(\lambda)^{\oplus m_{\mathcal{P},\lambda}},$$
\noindent will be written as,

\begin{equation}\label{decomp2}
\mathrm{Ind}_{J}^{K} \lambda = \bigoplus_{\sigma \in \widehat{K}} \sigma^{\oplus m_{\sigma}}.
\end{equation}
\noindent The integers $m_{\sigma}$ are the multiplicities of $\sigma$'s in $\mathrm{Ind}_{J}^{K} \lambda$. It follows that :

\begin{equation} \label{eq:1}
\mathrm{c\text{--} Ind}_{J}^{G} \lambda = \mathrm{c\text{--} Ind}_{K}^{G} (\mathrm{Ind}_{J}^{K} \lambda) = \bigoplus_{\sigma \in \widehat{K}} (\mathrm{c\text{--} Ind}_{K}^{G}\sigma)^{\oplus m_{\sigma}}.
\end{equation}
We will state now our main result:

\begin{thm}\label{0.1}
	Let $\sigma_1, \sigma_2 \in \widehat{K}$ with multiplicities $m_{\sigma_1}$, $m_{\sigma_2}$ respectively. Then $\mathrm{Hom}_G(\cI_K^G \sigma_1, \cI_K^G \sigma_2)$ is a  free $\mathfrak{Z}_{\Omega}$-module of rank $m_{\sigma_1}m_{\sigma_2}$. 
\end{thm}

As a corollary of the result above we get:

\begin{coro}
	Let $\sigma \in \widehat{K}$, such that $m_{\sigma}=1$. Then the canonical map $\mathfrak{Z}_{\Omega} \rightarrow \mathrm{End}_{G}(\mathrm{c\text{--} Ind}_{K}^{G} \sigma)$ induces a ring isomorphism:
	$$\mathfrak{Z}_{\Omega} \simeq \mathrm{End}_{G}(\mathrm{c\text{--} Ind}_{K}^{G} \sigma).$$
\end{coro}

\noindent this is the statement of Corollary \ref{coro2}. 

Let me finish this introduction by saying that the representation-theoretic results proven here will be used by \cite{Pyv1} to establish the Breuil-Schneider conjecture in many new cases, following the method of \cite{MR3529394}.

This paper is organised in the following manner. In section \ref{H.1} we will recall some facts about representations of $G$ and prove a few easy lemmas. Next, in section \ref{H.2}, we will prove some results about Bernstein centre. These results will allow us to study the specialization of a projective generator at maximal ideals that belong to some dense set. This will be achieved in section \ref{H.3}. Then in sections \ref{H.4} and \ref{H.5} we collect some technical results that will be needed in the next section. Then in section \ref{H.6} we will prove the main result of this paper. 

\subsection*{Notation}\label{notation}

Let $G$ be $p$-adic reductive group and let $E$ be an algebraically closed field of characteristic zero. Let  $\mathcal{R}(G)$ be the category of all smooth  $E$-representations of $G$.We denote by $i_{P}^{G} : \mathcal{R}(M) \longrightarrow \mathcal{R}(G)$ the normalized parabolic induction functor, where $P=MN$ is a parabolic subgroup of $G$ with Levi subgroup $M$. Let $\overline{P}$ be the opposite parabolic with respect to $M$. We use the notation $\mathrm{Ind}$ and $\cI$ to denote the induction and compact induction respectively.

The Bernstein decomposition expresses the category $\mathcal{R}(G)$ as the product of certain indecomposable full subcategories, called Bernstein components. These components are parametrized by the inertial classes. Let me now recall the definition of an inertial class. Let $M$ be a Levi subgroup of some parabolic subgroup of $G$ and let $\rho$ be an irreducible supercuspidal representation of $M$ and consider a set of pairs $(M, \rho)$ as above. We say that two pairs $(M_1, \rho_1)$ and $(M_2, \rho_2)$ are inertially equivalent if and only if there is $g \in G$ and an unramified character $\chi$ of $M_2$ such that, $M_2=M_1^{g}$ and $\rho_2 \simeq \rho_1^g \otimes\chi$, where $M_1^g:=g^{-1}M_1g$ and $\rho_1^g(x) = \rho_1(gxg^{-1})$, $\forall x \in M_1^g$. An equivalence class of $(M, \rho)$ will be denoted $[M,\rho]_{G}$. The set of all such equivalence classes will be denoted by $\mathcal{B}(G)$.

\medskip

In order to state the Bernstein decomposition, let me introduce some further notation. We are given an inertial class $\Omega:=[M,\rho]_{G}$, where $\rho$ is a supercuspidal representation of $M$ and $D:=[M,\rho]_{M}$. To any inertial class $\Omega$ we may associate a full subcategory $\mathcal{R}^{\Omega}(G)$ of $\mathcal{R}(G)$, such that $(\pi,V)$ is an object of $\mathcal{R}^{\Omega}(G)$ if and only if every irreducible $G$-subquotient $\pi_{0}$ of $\pi$ appear as a composition factor of $i_{P}^{G}(\rho \otimes \omega)$ for $\omega$ some unramified character of $M$ and $P$ some parabolic subgroup of $G$ with Levi factor $M$. The category $\mathcal{R}^{\Omega}(G)$ is called a Bernstein component of $\mathcal{R}(G)$. We will say that a representation $\pi$ is in $\Omega$ if $\pi$ is an object of $\mathcal{R}^{\Omega}(G)$. According to \cite{MR771671}, we have a decomposition:
\[\mathcal{R}(G) = \prod_{\Omega \in \mathcal{B}(G)} \mathcal{R}^{\Omega}(G).\]
So in order to understand the category $\mathcal{R}(G)$, it is enough to restrict our attention to the components. We may understand those components via the theory of types, developed by Bushnell and Kutzko. This theory allows us to parametrize all the irreducible representations of $G$ up to inertial equivalence using irreducible representations of compact open subgroups of $G$.

Let me briefly define Bushnell--Kutzko types. Let $J$ be a compact open subgroup of $G$ and let $\lambda$ be an irreducible representation of $J$. Let $K$ be a maximal compact open subgroup of $G$ containing $J$. We say that $(J, \lambda)$ is an $\Omega$-type if and only if for every irreducible object $(\pi,V)$ of $\mathcal{R}^{\Omega}(G)$, $V$ is generated by  the $\lambda$-isotypical component of $V$ as $G$-representation. 

\medskip

Let $F$ be a non-archimedean local field, with a finite residue field $k_{F}$. Let $\mathcal{O}_F$ be its complete discrete valuation ring, let $\mathfrak{p}$ be the maximal ideal of $\mathcal{O}_F$ with uniformizer $\varpi$, and let $q=|k_F|$. In this paper we only consider the case $G=GL_n(F)$.

For $G=GL_n(F)$, the types can be constructed in an explicit manner (cf. \cite{MR1204652}, \cite{MR1643417} and \cite{MR1711578}) for every Bernstein component. Moreover, Bushnell and Kutzko have shown that the Hecke algebra $\mathcal{H}(G, \lambda):= \mathcal{H}(G, J, \lambda):=\mathrm{End}_{G}(\cI_J^G \lambda)$ is naturally isomorphic to a tensor product of affine Hecke algebras of type A. The simplest example of a type is $(I, 1)$, where $I$ is Iwahori subgroup of $G$ and $1$ is the trivial representation of $I$. In this case $\Omega = [T,1]_G$, where $T$ is the subgroup of diagonal matrices. We will refer to example as the Iwahori case.

Let $\mathcal{R}_{\lambda}(G)$ be the full subcategory of $\mathcal{R}(G)$ such that $(\pi,V)$ is an object of $\mathcal{R}_{\lambda}(G)$ if and only if $V$ is generated by $V^{\lambda}$ (the $\lambda$-isotypical component of $V$) as $G$-representation.

\bigskip

Moreover every Bernstein component has a description as a category of modules over $\mathcal{H}(G, \lambda)$. Indeed, for any $\Omega$-type $(J,\lambda)$, by Theorem 4.2 (ii)\cite{MR1643417},  we have a functor:
$$\begin{array}{ccccc}
\mathfrak{M}_{\lambda} & : & \mathcal{R}_{\lambda}(G) & \to & \mathcal{H}(G,\lambda)-Mod \\
& & \pi &  \mapsto & \mathrm{Hom}_{J}(\lambda, \pi) = \mathrm{Hom}_{G}(\cI_J^G \lambda, \pi)\\
\end{array},$$

\noindent which is an equivalence of categories. Since  $(J,\lambda)$ is an $\Omega$-type, we have $\mathcal{R}^{\Omega}(G)=\mathcal{R}_{\lambda}(G)$.

\medskip

Denote by $W$ the vector space on which the representation $\lambda$ is realized. Next, let $(\check{\lambda},W^{\vee})$ denote the contragradient of $(\lambda,W)$.

Then by (2.6) \cite{MR1711578}, the Hecke algebra $\mathcal{H}(G, \lambda):=\mathrm{End}_G(\cI_J^G \lambda)$ can be identified with the space of compactly supported functions $f : G \longrightarrow \mathrm{End}_{E}(W^{\vee})$ such that $f(j_1.g.j_2) = \check{\lambda}(j_1)\circ f(g) \circ \check{\lambda}(j_2)$, with $j_1$, $j_2 \in J$ and $g \in G$ and the multiplication of two elements $f_1$ and $f_2$ is given by the convolution:
$$f_1*f_2(g) = \int_G f_1(x)\circ f_2(x^{-1}g) dx$$

For $u \in \mathrm{End}_E(W^{\vee})$, we write $\check{u} \in \mathrm{End}_{E}(W)$ for the transpose of $u$ with respect of the canonical pairing between $W$ and $W^{\vee}$. This gives $(\check{\lambda}(j))^{\vee} = \lambda(j)$, for $j \in J$. For $f \in \mathcal{H}(G,\lambda)$, define $\check{f} \in \mathcal{H}(G, \check{\lambda})$, by $\check{f}(g) = f(g^{-1})^{\vee}$, for all $g \in G$.

Write $\mathfrak{Z}_{\Omega}$ for the centre of category $\mathcal{R}^{\Omega}(G)$ and $\mathfrak{Z}_{D}$ for the centre of category $\mathcal{R}^{D}(M)$, which is defined the same way as $\mathcal{R}^{\Omega}(G)$. Recall that the centre of a category is the ring of endomorphisms of the identity functor. For example the centre of the category $\mathcal{H}(G,\lambda)-Mod$ is $Z(\mathcal{H}(G,\lambda))$, where $Z(\mathcal{H}(G,\lambda))$ is the centre of the ring $\mathcal{H}(G,\lambda)$. We will call $\mathfrak{Z}_{\Omega}$ a Bernstein centre.

\medskip

\section{Classical results and commutative algebra}\label{H.1}

We will start stating a few very useful results and we will introduce more notation. Recall that $\Omega:=[M,\rho]_{G}$, where $\rho$ is a supercuspidal representation of $M$ and $D:=[M,\rho]_{M}$. Write $\mathfrak{Z}_{\Omega}$ for the centre of category $\mathcal{R}^{\Omega}(G)$ and $\mathfrak{Z}_{D}$ for the centre of category $\mathcal{R}^{D}(M)$. Combining together theorem in section VI.4.4.(p.232) and the first lemma in section VI.10.3. (p.311) both in \cite{MR2567785}, we get following theorem:

\begin{thm} \label{thm0.1}
	$\mathcal{H}(G,\lambda)$ is a free finitely generated $\mathfrak{Z}_{D}$-module.
\end{thm}

The following result is the proved in section VI.10.3. \cite{MR2567785} (p.314), just before the statement of a theorem:

\begin{lemma} \label{lem0.1}
	$$\mathfrak{Z}_{\Omega}=\mathfrak{Z}_{D}^{W(D)}$$
	\noindent where
	$$W(D) = \left\lbrace g \in G | g^{-1}Mg=M \: \mathrm{and} \: [M,\rho^{g}]_{M}=D \right\rbrace / M.$$
\end{lemma}

Since $\mathfrak{Z}_D$ is an $E$-algebra of finite type we may use the results of \cite{MR782297} Chapitre 5 \S 1.9, to get:

\begin{lemma} \label{lem0.2}
	$\mathfrak{Z}_{D}$ is a $\mathfrak{Z}_{\Omega}$-module(algebra) of finite type.
\end{lemma}

Write $\chi$ for algebra homomorphism $\chi : \mathfrak{Z}_{\Omega} \rightarrow E$. Let $\mathfrak{m}=\mathrm{Ker}(\mathfrak{Z}_{\Omega} \xrightarrow{\chi} E)$ a maximal ideal of $\mathfrak{Z}_{\Omega}$ and $\kappa(\mathfrak{m})$ the residue field which is isomorphic to $E$. From now on we will always identify an algebra homomorphism $\chi : \mathfrak{Z}_{\Omega} \rightarrow E$ and a maximal ideal $\mathfrak{m}=\mathrm{Ker}(\mathfrak{Z}_{\Omega} \xrightarrow{\chi} E)$ of $\mathfrak{Z}_{\Omega}$.

\begin{lemma}\label{lem0.7}
	Let $A$ and $B$ be two $E$-algebras. Let $G$ a finite group acting on $A$ and $H$ another finite group acting on $B$, so that $G \times H$ acts on $A\otimes_{E} B$. Then the invariants under action of $G \times H$ are $(A\otimes_{E} B)^{G \times H} = (A^{G})\otimes_{E}( B^{H})$ 
\end{lemma}

\begin{lemma}\label{lem0.13}
	Let $A$ be a commutative $E$-algebra, which is also  a Jacobson ring. Let $f \in A$, which is a not a zero divisor in $A$. Then the set $\Spm(A[\frac{1}{f}])$ is Zariski dense in $\mathrm{Spec}(A)$.
\end{lemma}

\begin{lemma}\label{lem0.14}
	Let $Z:=E[X_{1},\ldots, X_{e}]$ and $S:=Z^{\mathfrak{S}_{e}}$, where the symmetric group $\mathfrak{S}_{e}$ acts by permutation of variables, i.e. $\sigma \in \mathfrak{S}_{e}$ acts by $\sigma.X_{i}=X_{\sigma(i)}$. Let $s_{i}:= \sum\limits_{1\leq j_{1}< \ldots <j_{i}\leq e} X_{j_{1}}\ldots X_{j_{i}}$ be the elementary symmetric polynomial, then $S \simeq E[s_{1},\ldots, s_{e}]$. Then $Z$ is a free $S$-module of rank $e!$ with basis given by monomials $X^{\nu}:= X_{1}^{\nu(1)}\ldots X_{e}^{\nu(e)}$, such that $0 \leq \nu(i) <i$ for $1 \leq i \leq e$. Let $\Delta =\det (\mathrm{tr}_{Z/S}(X^{\mu}.X^{\nu}))_{\mu,\nu}$ and let $d=\prod_{i<j}(X_{i}-X_{j})^{2}$. Then $\Delta$ is some power of $d$.
\end{lemma}
\begin{proof} According to \cite{MR1994218} IV.$\S$6.1 Theorem 1 c) $Z$ is a free $S$-module of rank $e!$. Let's first prove that $d$ is irreducible in $S$. Assume that $d=d_1d_2=(-1)^{e(e-1)/2}\prod_{i\neq j} (X_i - X_j)$ with $d_1$ and $d_2$ both in $S$ and have positive degree. Let $T=\{(i,j)|i\neq j\}$. Since $Z$ is an UFD, then by uniqueness of factorization we have:
	\[d_k=c_k\prod_{(i,j)\in T_k}(X_i-X_j),\]
	
	\noindent where $k\in \{1,2\}$ and $c_k \in E$.  The subsets $T_k$ of $T$ are such that $T_k\neq \emptyset$, $T_1\cup T_2 = T$ and $T_1\cap T_2=\emptyset$. Since $d_k \in S$, then for all $\sigma \in \mathfrak{S}_e$ we have that $\sigma.d_k=d_k$, then
	\[\prod_{(i,j)\in T_k}(X_{\sigma(i)}-X_{\sigma(j)})=\prod_{(i,j)\in T_k}(X_i-X_j), \]
	
	\noindent again by uniqueness of factorization in $Z$, we may identify factors on both sides. In particular we have that if $(i,j)\in T_k$ then for any permutation $\sigma$ we have that $(\sigma(i),\sigma(j)) \in T_k$. This implies that $T \subseteq T_k$, a contradiction.
	
	The map $f:\Sp Z \to \Sp S$ induced by an embedding $S\hookrightarrow Z$ is \'{e}tale at a point $x$ if and only if it is unramified at $x$. However the zero locus of $\Delta$, $V(\Delta)$, is equal to the set of points where the map $f$ is ramified (i.e. is not \'{e}tale), by definition of the discriminant. The map $f$ is not \'{e}tale when $X_i=X_j$ for $i\neq j$, this is the zero locus of $d$. Since $d$ is irreducible in $S$, it follows that $\Delta$ is some power of $d$.
\end{proof}

\section{Properties of Bernstein centre}\label{H.2}

Recall that $\Omega:=[M,\rho]_{G}$, where $\rho$ is a supercuspidal representation of $M$ and $D:=[M,\rho]_{M}$. Our goal is to determine $\mathfrak{Z}_{D}\otimes_{\mathfrak{Z}_{\Omega}} \kappa(\mathfrak{m})$ when $\mathfrak{m}$ varies through a dense set of maximal ideals in $\Sp \mathfrak{Z}_{\Omega}$, where $\mathfrak{Z}_{\Omega}$ for the centre of category $\mathcal{R}^{\Omega}(G)$ and $\mathfrak{Z}_{D}$ for the centre of category $\mathcal{R}^{D}(M)$. 

Let's first describe the action of $W(D)$ on  $\mathfrak{Z}_{D}$. Let $\mathcal{X}(M)$ be the group of unramified characters of $M$ and $\mathcal{X}(M)(\rho) = \left\lbrace \chi \in \mathcal{X}(M) | \rho \simeq \rho \otimes \chi \right\rbrace $. Let $M^{\circ}$ be the intersection of the kernels of the characters $\chi \in \mathcal{X}(M)$ and let $T$ be the intersection of the kernels of the $\chi \in \mathcal{X}(M)(\rho)$. The restriction  to $T$ induces a bijection $\mathcal{X}(M)/ \mathcal{X}(M)(\rho) \simeq \mathcal{X}(T)$. Let $Irr(D)$ be the set of irreducible representations in $D$. Every such a representation is of the form $\rho \otimes \chi$ for $\chi \in \mathcal{X}(M)$. Thus we have a bijection $\mathcal{X}(M)/ \mathcal{X}(M)(\rho) \simeq Irr(D)$, $ \chi \mapsto \rho \otimes \chi$. Composing it with previous bijection we get a bijection $Irr(D) \simeq \mathcal{X}(T)$. Now $\mathcal{X}(T)$ is naturally isomorphic to the set of $E$-algebra homomorphisms from $E[T/M^{\circ}]$ to $E$. It is explained in \cite{MR2567785} section V.4.4 that we have an identification $\mathfrak{Z}_{D} \simeq E[T/M^{\circ}]$, so that the bijection $Irr(D) \simeq \mathcal{X}(T)$ induces a natural bijection between $Irr(D)$ and $\Spm(\mathfrak{Z}_{D})$. The group $W(D)$ acts on $Irr(D)$ by conjugation. For each $w \in W(D)$ let $\xi_w \in \mathcal{X}(M)$ be any character such that $\rho^{w} \simeq \rho \otimes \xi_w$. We will use the same symbol $\xi_{w}$ for the restriction of $\xi_w$ to $T$. If $\chi \in \mathcal{X}(M)$ then $(\rho \otimes \chi)^{w} \simeq \rho \otimes \chi^{w}.\xi_w$. Thus the action of $W(D)$ on $\mathcal{X}(T)$ is given by $w.\chi = \chi^{w}.\xi_{w}$. It is immediate that the induced action on $E[T/M^{\circ}] \simeq\mathfrak{Z}_{D}$ is given by $w.(tM^{\circ}) =\xi_{w}(t)^{-1}t^{w}M^{\circ}$.

\begin{lemma}\label{lem0.15}
	An $E$-algebra homomorphism $\mathrm{X} :\mathfrak{Z}_D \to E$ can be lifted to an unramified character $\overline{\chi}$ of $M$, i.e. we have a surjective map:
	\[ \mathcal{X}(M) \twoheadrightarrow \mathrm{Hom}_{E-alg}(\mathfrak{Z}_{D},E). \]
	This map has the following description, given an unramified character $\overline{\chi}$ of $M$, we can associate to it a $E$-algebra homomorphism $\mathrm{X} :\mathfrak{Z}_D \to E$, defined as:
	$$\begin{array}{ccccc}
	\mathrm{X} & : & \mathfrak{Z}_{D} & \to & E \\
	& & z &  \mapsto & z(\overline{\chi}), \\
	\end{array}$$ 
	
	\noindent where $z(\overline{\chi})$ is a scalar by which $z$ acts on one dimensional representation $\overline{\chi}$ of $M$.

\end{lemma}

\begin{proof} By the description of the action of $W(D)$ on  $\mathfrak{Z}_{D}$, above this lemma, we have the following isomorphisms:
	\[\mathcal{X}(M)/ \mathcal{X}(M)(\rho) \simeq \mathcal{X}(T) \simeq \mathrm{Hom}_{E-alg}(\mathfrak{Z}_{D},E),\]
	
	\noindent hence a surjective map $\mathcal{X}(M) \longrightarrow \mathrm{Hom}_{E-alg}(\mathfrak{Z}_{D},E)$.
\end{proof}

Let $\mathfrak{m}=\mathrm{Ker}(\mathfrak{Z}_{\Omega} \xrightarrow{\overline{\chi}} E)$ a maximal ideal of $\mathfrak{Z}_{\Omega}$ and $\kappa(\mathfrak{m})$ the residue field which is isomorphic to $E$, because $E$ is algebraically closed.

\begin{lemma} \label{lem0.5}
	There is a dense set in $\Spm \mathfrak{Z}_{\Omega}$ of maximal ideals $\mathfrak{m} \in \Spm \mathfrak{Z}_{\Omega}$, such that:
	$$\mathfrak{Z}_{D}\otimes_{\mathfrak{Z}_{\Omega}} \kappa(\mathfrak{m}) \simeq \prod_{k=1}^{|W(D)|}\kappa(\mathfrak{M}_{i}),$$
	
	\noindent where $\mathfrak{M}_{i}$ are maximal ideals of $\mathfrak{Z}_{D}$ above $\mathfrak{m}$, and $\kappa(\mathfrak{M}_{i})$ the residue fields. Moreover $\mathfrak{Z}_{D}$ is free over $\mathfrak{Z}_{\Omega}$ of rank $|W(D)|$.
\end{lemma}

\begin{proof} Let's first deal with two particular cases before dealing with general case.
	
	\noindent \textbf{1. Supercuspidal case. } In this case we have $M=G$, then $\mathfrak{Z}_{D} \simeq \mathfrak{Z}_{\Omega}$. Everything is clear, there is nothing to prove.
	
	\medskip
	
	\noindent \textbf{2. Simple type case. } Assume now that $(J, \lambda)$ is a simple type, without loss of generality we may assume then $M=GL_{k}(F)^{e}$ and $\rho = \pi \otimes \ldots \otimes \pi$ ($e$ times), where $\pi$ is a supercuspidal representation of $GL_{k}(F)$.
	
	By Theorem (6.6.2) \cite{MR1204652}, there is a maximal type $(J_{0}, \lambda_{0})$ of $GL_{k}(F)$,  a field extension $\Gamma$ of $F$ and a uniquely determined representation $\Lambda$ of $\Gamma^{\times}J_{0}$ such that $\Lambda| J_{0} = \lambda_{0}$ and $\pi = \mathrm{c\text{--} Ind}_{\Gamma^{\times}J_{0}}^{GL_{k}(F)}\Lambda$. From Frobenius reciprocity it follows that there is a Hecke algebra isomorphism $\mathcal{H}(M, \lambda_{M}) \simeq \mathcal{H}(\widetilde{J_{M}}, \widetilde{\lambda_{M}})$, because any $g \in M$ that intertwines $\lambda_{M}$ lies in $\widetilde{J_{M}}$, where $\widetilde{J_{M}}$ is a subgroup of $M$ compact modulo centre of $M$ such that $J_M=\widetilde{J_{M}} \cap K \cap M$ and $J_M:=J\cap M$. Since $\widetilde{J_{M}}/J_{M}$ is a free abelian group, $\mathcal{H}(\widetilde{J_{M}}, \widetilde{\lambda_{M}})$ is commutative, and we have an isomorphism $\mathcal{H}(\widetilde{J_{M}}, \widetilde{\lambda_{M}}) \simeq E[\widetilde{J_{M}}/J_{M}]$. Therefore we have:
	\begin{align}
	\mathfrak{Z}_{D} & \simeq \mathcal{H}(M, \lambda_{M}) \simeq \mathcal{H}(\widetilde{J_{M}}, \widetilde{\lambda_{M}}) \simeq E[\widetilde{J_{M}}/J_{M}]  \simeq E [(\Gamma^{\times}J_{0})^{e}/J_{0}^{e}] \nonumber \\
	& \simeq E [(\Gamma^{\times}J_{0}/J_{0})^{e}] \simeq E [(\Gamma^{\times}/\mathcal{O}_{\Gamma}^{\times})^{e}] \simeq E [((\varpi_{\Gamma})^{\Z})^{e}] \simeq E[X_{1}^{\pm 1}, \ldots, X_{e}^{\pm 1}], \nonumber
	\end{align}

	\noindent where $\mathcal{O}_{\Gamma}$ is the ring of integers of $\Gamma$ and $\varpi_{\Gamma}$ an uniformizer. Since $\rho = \pi \otimes \ldots \otimes \pi$, it follows from the description action of $W(D)$ on $\mathfrak{Z}_{D}$ in the beginning of this section, that the characters $\xi_w$ are trivial, i.e. $\xi_w=1$ for all $w \in W(D)$. Then  
	the group  $W(D) \simeq \mathfrak{S}_{e}$ acts by permutation of variables $X_{i}$ on $\mathfrak{Z}_{D} \simeq E[X_{1}^{\pm 1}, \ldots, X_{e}^{\pm 1}]$.
	
	Let $Z:=E[X_{1},\ldots, X_{e}]$ and $S:=Z^{\mathfrak{S}_{e}}$. Let $s_{i}:= \sum\limits_{1\leq j_{1}< \ldots <j_{i}\leq e} X_{j_{1}}\ldots X_{j_{i}}$ be the $i$-th elementary symmetric polynomial,then $S \simeq E[s_{1},\ldots, s_{e}]$. It follows that 
	\begin{align}
	\mathfrak{Z}_{\Omega} & =  \mathfrak{Z}_{D}^{W(D)} \simeq E[X_{1}^{\pm 1}, \ldots, X_{e}^{\pm 1}]^{\mathfrak{S}_{e}} \simeq E[X_{1}, \ldots, X_{e}]^{\mathfrak{S}_{e}} \otimes_E E[X_{1}^{-1}, \ldots, X_{e}^{-1}]^{\mathfrak{S}_{e}} \nonumber \\ 
	& \simeq E[s_{1}, \ldots, s_{e}] \otimes_E E[s'_{1}, \ldots, s'_e] \simeq E[s_{1}, \ldots, s_{e}, s'_{1}, \ldots,s'_{e-1} s_e^{-1}]\nonumber \\ 
	& \simeq E[s_{1}, \ldots,s_{e-1}, s_{e}^{\pm 1}],\nonumber
	\end{align}

	\noindent where $s'_i$ are symmetric polynomials in the variable $X_j^{-1}$. The last isomorphism follows from the fact that $s'_i=s_{e-i}/s_e$, for $0\leq i \leq e-1$, and $s'_e=s_e^{-1}$. After a localization with respect of $\left\lbrace s_{e}^{n}\right\rbrace_{n \geq 0}$  we see that $\mathfrak{Z}_{D} = Z_{s_{e}}$ is a free $\mathfrak{Z}_{\Omega} = S_{s_{e}}$-module of rank $|W(D)|=e!$ with basis given by monomials $X^{\nu}:= X_{1}^{\nu(1)}\ldots X_{e}^{\nu(e)}$, such that $0 \leq \nu(i) <i$ for $1 \leq i \leq e$, by Lemma \ref{lem0.14}. Let $d=\prod_{i<j}(X_{i}-X_{j})^{2} \in \mathfrak{Z}_{\Omega}$. By Lemma \ref{lem0.14} the discriminant is some power of $d$.
	
	When the specialization $d(\mathfrak{m}):=d \otimes \kappa(\mathfrak{m})$ of $d$ at a maximal ideal $\mathfrak{m}$ is non zero, then $\mathfrak{m}$ is of form $(s_{1}-a_{1},\ldots, s_{e}-a_{e}) $, where the $a_{1}$,...,$a_{e}$ are such that the polynomial $f \in \kappa(\mathfrak{m})[X]$ defined by $f=X^{e}+\sum\limits_{k=1}^{e}(-1)^{k}a_{k}X^{e-k}$ has $e$ distinct roots, say $\alpha_{1}$,..., $\alpha_{e}$. Let $w \in W(D)\simeq \mathfrak{S}_{e}$, set $\mathfrak{M}_{w}$ the kernel of homomorphism $\mathfrak{Z}_{D} \longrightarrow E$ sending $X_{k} \mapsto \alpha_{w(k)}$. Moreover $\mathfrak{M}_{w}$ is a maximal ideal of $\mathfrak{Z}_{D}$ above $\mathfrak{m}$. We have a natural surjection :
	$$\mathfrak{Z}_{D}\otimes_{\mathfrak{Z}_{\Omega}} \kappa(\mathfrak{m})\twoheadrightarrow \prod_{w \in W(D)} \kappa(\mathfrak{M}_{w}).$$
	
	Since $\dim_{\kappa(\mathfrak{m})}(\mathfrak{Z}_{D}\otimes_{\mathfrak{Z}_{\Omega}} \kappa(\mathfrak{m})) = |W(D)|$, this surjection is an isomorphism of $E$ vector spaces by comparing the dimensions. Then $\mathfrak{Z}_{D}\otimes_{\mathfrak{Z}_{\Omega}} \kappa(\mathfrak{m})$ is a product of $|W(D)|$ copies of $E$, since $E$ is assumed to be algebraically closed. 
	
	Moreover, the set $S:= \left\lbrace \mathfrak{m} \in \Spm(\mathfrak{Z}_{\Omega}) | d(\mathfrak{m}) \neq 0 \right\rbrace = \Spm(\mathfrak{Z}_{\Omega}[\frac{1}{d}])$ is not empty and Zariski dense, because of the Lemma \ref{lem0.13}.
	\medskip
	
	\noindent \textbf{3. General case. } Now let's treat the general case, where the type $(J, \lambda)$ is semi-simple. We may always assume that $M=\prod_{i=1}^{s} M_{i}$ and $\rho = \bigotimes_{i=1}^{s} \rho_{i}$, where $M_{i}=GL_{n_{i}}(F)^{e_{i}}$ and $\rho_{i}$ is a supercuspidal representation of $M_{i}$. Define $G_{i} = GL_{n_{i}e_{i}}(F)$, $\Omega_{i} = [M_{i}, \rho_{i}]_{G_{i}}$, $D_{i} = [M_{i}, \rho_{i}]_{M_{i}}$. 
	Let $\overline{M}$  be the unique Levi subgroup of $G$ which contains the $N_{G}(M)$-stabilizer of the inertia class $D$ and is minimal for this property. Section 1.5 in \cite{MR1711578} applied to $\mathcal{H}(\overline{M}, \lambda_{M}) \simeq \mathfrak{Z}_{D}$ gives:
	$$\mathfrak{Z}_{D} \simeq \bigotimes_{i=1}^{s}\mathfrak{Z}_{D_{i}}$$
	
	\noindent and 
	$$W(D) \simeq \prod_{i=1}^{s} W(D_{i}).$$
	
	\noindent The action of $W(D)$ on $\mathfrak{Z}_{D}$ is such that every  $W(D_{i})$ acts only on $\mathfrak{Z}_{D_{i}}$.
	
	\noindent An inductive application of Lemma \ref{lem0.7}, to the previous decomposition of $\mathfrak{Z}_{D}$ gives:
	$$\mathfrak{Z}_{\Omega} \simeq \bigotimes_{i=1}^{s}\mathfrak{Z}_{\Omega_{i}}.$$
	
	\noindent By previous case we have the following non canonical isomorphisms :
	$$\mathfrak{Z}_{D_{i}} \simeq E[X_{1,i}^{\pm 1}, \ldots, X_{e_{i},i}^{\pm 1}],$$
	$$\mathfrak{Z}_{\Omega_{i}} \simeq E[X_{1,i}^{\pm 1}, \ldots, X_{e_{i},i}^{\pm 1}]^{\mathfrak{S}_{e_{i}}}.$$
	
	\noindent We have $W(D_{i}) \simeq \mathfrak{S}_{e_{i}}$ and we may assume that $\mathfrak{S}_{e_{i}}$ acts on $E[X_{1,i}^{\pm 1}, \ldots, X_{e_{i},i}^{\pm 1}]$ by permutation, since it is always the case after an appropriate linear change of variables. Moreover $\mathfrak{Z}_{D_{i}}$ is free $\mathfrak{Z}_{\Omega_{i}}$-module of rank $e_{i}!$, and let  $d_{e_{i}}=\prod_{k<l}(X_{k,i}-X_{l,i})^{2}$.
	
	Then $\mathfrak{Z}_{D}$ is free $\mathfrak{Z}_{\Omega}$ module of rank $|W(D)|$ and define $d = \prod_{i=1}^{s} d_{e_{i}}$.  The proof of general case ends exactly in the same way as in the simple type case and the set $S':=\Spm(\mathfrak{Z}_{\Omega}[\frac{1}{d}])$ is not empty and Zariski dense, by Lemma~\ref{lem0.13}.
\end{proof}

\begin{lemma} \label{lem0.3}
	$\mathcal{H}(G,\lambda)$ is a free and finitely generated $\mathfrak{Z}_{\Omega}$-module. 
\end{lemma}

\begin{proof} It follows from the proof of Lemma \ref{lem0.5}, that $\mathfrak{Z}_{D}$ is free $\mathfrak{Z}_{\Omega}$ module of rank $|W(D)|$. Finally Theorem \ref{thm0.1} gives the desired result.
\end{proof}

\noindent \textbf{Remark.} Lemma \ref{lem0.3} above is essentially the same as Lemma 2.1 in \cite{MR1676870}.

\section{Specialization of a projective generator at maximal ideal of Bernstein centre}\label{H.3}

In this section we compute $\mathrm{c\text{--} Ind}_{J}^{G}\lambda \otimes_{\mathfrak{Z}_{\Omega}} \kappa(\mathfrak{m})$ for $\mathfrak{m} \in \Sp \mathfrak{Z}_{\Omega}$ a maximal ideal which belongs to some dense set of points in $\Sp \mathfrak{Z}_{\Omega}$, where $\Omega:=[M,\rho]_{G}$,  $D:=[M,\rho]_{M}$, and $\mathfrak{Z}_{\Omega}$ for the centre of category $\mathcal{R}^{\Omega}(G)$. This result is an improvement of Lemma 1.2 in \cite{MR1676870}. Let $(J, \lambda)$ be a Bushnell-Kutzko $\Omega$-type, such that $J$ is contained in a maximal compact subgroup $K$. The representation $\mathrm{c\text{--} Ind}_{J}^{G}\lambda$ is a projective generator of $\mathcal{R}^{\Omega}(G)$. 

Let $\overline{\chi}$ be any lift of $\mathrm{X} : \mathfrak{Z}_D \rightarrow E$ as in Lemma \ref{lem0.15}. Let now $\chi=\mathrm{X}|\mathfrak{Z}_{\Omega}$. We say that a character $\chi$ on $\mathfrak{Z}_{\Omega}$, obtained in this way, is induced from unramified character $\overline{\chi}$ of $M$. 

Once and for all we fix the following notation. For $(J, \lambda)$ an $\Omega$-type, as above, there exists a $D$-type $(J_{M}, \lambda_{M})$, such that :
\begin{itemize}
	\item[1.] $J_{M} = J \cap M$ and $\lambda_{M} = \lambda | J_{M}$.
	\item[2.] $J$ has an Iwahori decomposition $J \simeq (J \cap \overline{N})(J \cap M)(J \cap N)$ such that $\lambda |(J \cap \overline{N})$ and $\lambda |(J \cap N)$ are trivial. Here $\overline{N}$ is unipotent radical of opposite parabolic subgroup $\overline{P}=M\overline{N}$.
	\item[3.] For any  parabolic subgroup $P$ with Levi component $M$, there is an  there is an invertible element in $\mathcal{H}(G,\lambda)$ supported in $Jz_{P}J$, such that the element $z_{P}$ is in the centre of $M$ and has the following properties: 
	\begin{itemize}
		\item $z_P(J \cap N)z_P^{-1} \subset J \cap N$. 
		\item $z_P^{-1}(J \cap \overline{N})z_P \subset J \cap \overline{N}$.
		\item For any compact open subgroups $H_1$, $H_2$ of $N$ there is an integer $m\geq 0$, such that $z_P^{m} H_1 z_P^{-m} \subset H_2$.
		\item For any compact open subgroups $\overline{H}_2$, $\overline{H}_2$ of $\overline{N}$ there is an integer $m\geq 0$, such that $z_P^{m} \overline{H}_1 z_P^{-m} \subset \overline{H}_2$.
	\end{itemize}
	\item[4.] There is a subgroup $\widetilde{J_{M}}$ of $M$ compact modulo centre of $M$ such that $J_{M} = \widetilde{J_{M}} \cap K \cap M$.
	
	\item[5.] There is an extension $\widetilde{\lambda_{M}}$  of $\lambda_{M}$ to $\widetilde{J_{M}}$ such that  $\rho = \mathrm{c\text{--} Ind}_{\widetilde{J_{M}}}^{M} \widetilde{\lambda_{M}}$ (is irreducible supercupidal) and any $g \in M$ which intertwines $\lambda_{M}$ lies in $\widetilde{J_{M}}$.
	
\end{itemize}

\medskip
The theorem on existence of $G$-covers(section (8.3) \cite{MR1643417}) ensures the conditions 1, 2, 3. The conditions 4 and 5 follow from section (5.5)\cite{MR1643417}. Now we state and prove the main result of this section:

\begin{prop} \label{lem0.6}
	Let $\chi : \mathfrak{Z}_{\Omega} \rightarrow E$ be an algebra homomorphism corresponding to a maximal ideal $\mathfrak{m}=\mathrm{Ker}(\mathfrak{Z}_{\Omega} \xrightarrow{\chi} E)$ of $\mathfrak{Z}_{\Omega}$. Then there is a Zariski dense set $S$ in $\mathrm{Spec}(\mathfrak{Z}_{\Omega})$ such that:
	$$\mathrm{c\text{--} Ind}_{J}^{G}\lambda \otimes_{\mathfrak{Z}_{\Omega}} \kappa(\mathfrak{m}) \simeq P(\chi)^{\oplus |W(D)|}$$
	
	\noindent$\forall \mathfrak{m} \in S$,  where $P(\chi):=i_{\overline{P}}^{G}(\rho \otimes \overline{\chi})$ is an irreducible parabolic induction of a supercuspidal representation of a Levi subgroup of $G$ and $\overline{\chi}$ some character corresponding to the algebra homomorphism $\mathrm{X} : \mathfrak{Z}_{D} \rightarrow E$, as in Lemma \ref{lem0.15}, such that $\mathfrak{M} = \mathrm{Ker}(\mathrm{X})$ is a maximal ideal of $\mathfrak{Z}_{D}$ above $\mathfrak{m}$.
\end{prop}

\begin{proof}  According to section 1.5 \cite[section 1.5]{MR1670599}.  We have the following isomorphism:
	
	\begin{equation} \label{eq:2}
	i_{\overline{P}}^{G}(\mathrm{c\text{--} Ind}_{J_{M}}^{M} \lambda_{M}) \simeq \mathrm{c\text{--} Ind}_{J}^{G} \lambda.
	\end{equation}
	
	\noindent The functor $i_{\overline{P}}^{G}$ is exact, hence:
	$$\mathrm{c\text{--} Ind}_{J}^{G} \lambda \otimes_{\mathfrak{Z}_{\Omega}} \kappa(\mathfrak{m}) \simeq i_{\overline{P}}^{G}(\mathrm{c\text{--} Ind}_{J_{M}}^{M} \lambda_{M} \otimes_{\mathfrak{Z}_{\Omega}} \kappa(\mathfrak{m}))$$
	
	\noindent by the previous isomorphism of representations. Let's find a decomposition of $\mathrm{c\text{--} Ind}_{J_{M}}^{M} \lambda_{M} \otimes_{\mathfrak{Z}_{\Omega}} \kappa(\mathfrak{m})$. Indeed
	$$\mathrm{c\text{--} Ind}_{J_{M}}^{M} \lambda_{M} =\mathrm{c\text{--} Ind}_{\widetilde{J_{M}}}^{M} \mathrm{ c\text{--}Ind}_{J_{M}}^{\widetilde{J_{M}}} \lambda_{M} = \mathrm{c\text{--} Ind}_{\widetilde{J_{M}}}^{M} (\widetilde{\lambda_{M}} \otimes \mathrm{ c\text{--}Ind}_{J_{M}}^{\widetilde{J_{M}}} 1).$$
	
	Let us now describe the action of $\mathfrak{Z}_D$. From Frobenius reciprocity follows a Hecke algebras isomorphism $\mathcal{H}(M, \lambda_{M}) \simeq \mathcal{H}(\widetilde{J_{M}}, \widetilde{\lambda_{M}})$ because any $g \in M$ that intertwines $\lambda_{M}$ lies in $\widetilde{J_{M}}$. Since $\widetilde{J_{M}}/J_{M}$ is free abelian group, $\mathcal{H}(\widetilde{J_{M}}, \widetilde{\lambda_{M}})$ is commutative, and we have an isomorphism $\mathcal{H}(\widetilde{J_{M}}, \widetilde{\lambda_{M}}) \simeq E[\widetilde{J_{M}}/J_{M}]$. Therefore we have:
	$$\mathfrak{Z}_{D} \simeq \mathcal{H}(M, \lambda_{M}) \simeq \mathcal{H}(\widetilde{J_{M}}, \widetilde{\lambda_{M}}) \simeq E[\widetilde{J_{M}}/J_{M}] \simeq \mathrm{End}_{\widetilde{J_{M}}}(\mathrm{ c\text{--}Ind}_{J_{M}}^{\widetilde{J_{M}}} 1).$$
	
	The representation $\mathrm{ c\text{--}Ind}_{J_{M}}^{\widetilde{J_{M}}} 1$ is naturally isomorphic to the space of functions on $\widetilde{J_{M}}$ which are left invariant by $J_{M}$. We have the following canonical isomorphism $\mathrm{ c\text{--}Ind}_{J_{M}}^{\widetilde{J_{M}}} 1 \simeq E[\widetilde{J_{M}}/J_{M}]$. This shows that $\mathrm{ c\text{--}Ind}_{J_{M}}^{\widetilde{J_{M}}} 1$ is free $\mathfrak{Z}_{D}$-module of rank 1.
	
	\medskip
	
	Since $\mathrm{c\text{--} Ind}_{J_{M}}^{M} \lambda_{M} \otimes_{\mathfrak{Z}_{\Omega}} \kappa(\mathfrak{m}) \simeq \mathrm{c\text{--} Ind}_{J_{M}}^{M} \lambda_{M} \otimes_{\mathfrak{Z}_{D}}(\mathfrak{Z}_{D} \otimes_{\mathfrak{Z}_{\Omega}} \kappa(\mathfrak{m}))$, it is enough to find the decomposition of $\mathrm{c\text{--} Ind}_{J_{M}}^{M} \lambda_{M} \otimes_{\mathfrak{Z}_{D}} \kappa(\mathfrak{M}_{j})$ ($1 \leq j \leq |W(D)|$), because of Lemma \ref{lem0.5}. The functor $\mathrm{c\text{--} Ind}_{\widetilde{J_{M}}}^{M}$ is exact, therefore using projection formula:
	
	$$\mathrm{c\text{--} Ind}_{J_{M}}^{M} \lambda_{M} \otimes_{\mathfrak{Z}_{D}} \kappa(\mathfrak{M}_{i}) = \mathrm{c\text{--} Ind}_{\widetilde{J_{M}}}^{M} (\widetilde{\lambda_{M}} \otimes (\mathrm{ c\text{--}Ind}_{J_{M}}^{\widetilde{J_{M}}} 1)\otimes_{\mathfrak{Z}_{D}} \kappa(\mathfrak{M}_{j})).$$
	
	Let's express $\mathrm{ c\text{--}Ind}_{J_{M}}^{\widetilde{J_{M}}} 1\otimes_{\mathfrak{Z}_{D}} \kappa(\mathfrak{M}_{j})$ in terms of more suitable data. Let's drop the index $j$ temporarily and write $\mathfrak{M}:=\mathfrak{M}_{j}$. 
	
	Let now $\mathfrak{M}$ be a maximal ideal of $\mathfrak{Z}_{D}$ above some maximal ideal $\mathfrak{m} \in \Sp \mathfrak{Z}_{\Omega}$. We may always assume that $M=\prod_{i=1}^{s} M_{i}$ and $\rho = \bigotimes_{i=1}^{s} \rho_{i}$, where $M_{i}=GL_{n_{i}}(F)^{e_{i}}$ and $\rho_{i}\simeq \pi_{i} \otimes\ldots\otimes \pi_{i}$ ($e_{i}$ times) is a supercuspidal representation of $M_{i}$ and $\pi_{i}$ is a supercuspidal representation of $GL_{n_{i}}(F)$. Define $G_{i} = GL_{n_{i}e_{i}}(F)$, $\Omega_{i} = [M_{i}, \rho_{i}]_{G_{i}}$, $D_{i} = [M_{i}, \rho_{i}]_{M_{i}}$. Then: 
	$$\mathfrak{Z}_D\simeq E[X_{1,1}^{\pm 1}, \ldots, X_{e_{1},1}^{\pm 1},\ldots,X_{1,s}^{\pm 1}, \ldots, X_{e_{s},s}^{\pm 1}].$$
	
	\noindent Let
	$$\mathfrak{m}=(s_{1,i}-a_{1,i},\ldots,s_{e_{i},i}-a_{e_{i},i})_{1\leq i \leq s},$$
	
	\noindent where $s_{k,i}$ are elementary symmetric functions in variables $X_{1,i}$,\ldots,$X_{e_{i},i}$ and $a_{k,i} \in E$. Then
	$$\mathfrak{M}=(X_{1,i}-\alpha_{1,i},\ldots,X_{e_{i},i}-\alpha_{e_{i},i})_{1\leq i \leq s},$$
	
	\noindent where  for each $i$, $\alpha_{1,i}$,\ldots, $\alpha_{e_{i},i}$ are the $e_{i}$ distinct roots of  polynomial $X^{e_{i}}+\sum\limits_{k=1}^{e_{i}}(-1)^{k}a_{k}X^{e_{i}-k}$. We assumed that the extension $E$ is algebraically closed, so all those roots lie in $E$. Let $\overline{\chi}:=\overline{\chi}_j$, the unramified character which corresponds to $\mathfrak{M}$. Then $\overline{\chi} = \bigotimes_{i=1}^{s} \overline{\psi}_{i}$, where $\overline{\psi}_{i}$ are unramified characters of $M_{i} = GL_{n_{i}}(F)^{e_{i}}$, such that $\overline{\psi}_{i} =\bigotimes_{k=1}^{e_{i}} \overline{\psi}_{k,i}$ and if $\varpi$ denotes the uniformizer of $F$ and $I$ the identity matrix of $GL_{n_{i}}(F)$, $\overline{\psi}_{k,i}(\varpi.I)=\alpha_{k,i}$.
	
	Since $\mathrm{ c\text{--}Ind}_{J_{M}}^{\widetilde{J_{M}}} 1$ is a free $\mathfrak{Z}_D$-module of rank one, we have an isomorphism of $\mathfrak{Z}_D$-modules: 
	
	\begin{equation} \label{eq:4}
	\mathrm{ c\text{--}Ind}_{J_{M}}^{\widetilde{J_{M}}} 1\otimes_{\mathfrak{Z}_{D}} \kappa(\mathfrak{M}_{j}) \simeq \mathfrak{Z}_{D} / \mathrm{Ker}(\mathrm{X}_{j}) \simeq \mathrm{Im}(\mathrm{X}_{j})
	\end{equation}
	
	\noindent where $\mathfrak{M}_{i} = \mathrm{Ker}(\mathrm{X}_{j})$ and the algebra homomorphism $\mathrm{X}_{j} : \mathfrak{Z}_{D} \rightarrow E$ is such that the unramified character $\overline{\chi}_{j}$ of $M$ maps to $X_j$ as in Lemma \ref{lem0.15}. It follows from previous description of the maximal ideal $\mathfrak{M}:=\mathfrak{M}_{j}$ and the character $\overline{\chi}_{j}$ that:
	$$\mathrm{Im}(\mathrm{X}_{j}) = \mathrm{Im}(\overline{\chi}_{j}).$$
	
	\noindent Then from (\ref{eq:4}) follows that the representation $\mathrm{ c\text{--}Ind}_{J_{M}}^{\widetilde{J_{M}}} 1\otimes_{\mathfrak{Z}_{D}} \kappa(\mathfrak{M}_{j})$ is one dimensional and also we have an isomorphism of $\widetilde{J_{M}}$-representations:
	$$\mathrm{ c\text{--}Ind}_{J_{M}}^{\widetilde{J_{M}}} 1 \otimes_{\mathfrak{Z}_{D}} \kappa(\mathfrak{M}_{j}) \simeq \overline{\chi}_{j}|\widetilde{J_{M}}.$$
	
	\noindent Now using projection formula and previous isomorphism we may write:
	$$\mathrm{c\text{--} Ind}_{\widetilde{J_{M}}}^{M} (\widetilde{\lambda_{M}} \otimes (\mathrm{ c\text{--}Ind}_{J_{M}}^{\widetilde{J_{M}}} 1)\otimes_{\mathfrak{Z}_{D}} \kappa(\mathfrak{M}_{i})) \simeq \mathrm{c\text{--} Ind}_{\widetilde{J_{M}}}^{M} (\widetilde{\lambda_{M}} \otimes \overline{\chi}_{i}|\widetilde{J_{M}}) \simeq \rho \otimes \overline{\chi}_{j},$$
	
	\noindent because $\rho = \mathrm{c\text{--} Ind}_{\widetilde{J_{M}}}^{M} \widetilde{\lambda_{M}}$. So that we have
	
	\begin{equation} \label{eq:3}
	\mathrm{c\text{--} Ind}_{J_{M}}^{M} \lambda_{M} \otimes_{\mathfrak{Z}_{D}} \kappa(\mathfrak{M}_{i}) \simeq \rho \otimes \overline{\chi}_{i}
	\end{equation}

	\noindent Using (\ref{eq:2}) and (\ref{eq:3}) we get:
	$$ \mathrm{c\text{--} Ind}_{J}^{G} \lambda \otimes_{\mathfrak{Z}_{\Omega}} \kappa(\mathfrak{m}) \simeq i_{\overline{P}}^{G}(\mathrm{c\text{--} Ind}_{J_{M}}^{M} \lambda_{M} \otimes_{\mathfrak{Z}_{D}}(\mathfrak{Z}_{D} \otimes_{\mathfrak{Z}_{\Omega}} \kappa(\mathfrak{m})))$$ $$\simeq i_{\overline{P}}^{G}(\bigoplus_{j=1}^{|W(D)|} \mathrm{c\text{--} Ind}_{J_{M}}^{M} \lambda_{M} \otimes_{\mathfrak{Z}_{D}} \kappa(\mathfrak{M}_{j})) \simeq \bigoplus_{j=1}^{|W(D)|} i_{\overline{P}}^{G}(\rho \otimes \overline{\chi}_{j}),$$
	
	\noindent where the maximal ideal $\mathfrak{m}$ belongs to open dense set $S'$, defined by $S':=\Spm(\mathfrak{Z}_{\Omega}[\frac{1}{d}])$ as in Lemma \ref{lem0.5}.
	
	Let's now prove that all the $i_{\overline{P}}^{G}(\rho \otimes \overline{\chi}_{j})$ are irreducible on the subset $S:=\Spm(\mathfrak{Z}_{\Omega}[\frac{1}{d\Delta}])$ of $S'$, with $\Delta := \prod_{(k',i')\neq (k,i)}(X_{k',i'}-qX_{k,i})(X_{k,i}-qX_{k',i'})$, for all $1\leq k ,k'\leq e_{i}$ and $1 \leq i, i' \leq s$, and $q$ is the cardinality of the residue field of $F$. Again by the Lemma \ref{lem0.13} the set $S$ is dense. Let $\mathfrak{M}$ a maximal ideal of $\mathfrak{Z}_{D}$ above $\mathfrak{m} \in S$ corresponding to $\overline{\chi}_{j}$. With the same notation as above, we have then:
	$$\rho \otimes \overline{\chi}_{j} = \bigotimes_{i=1}^{s} \bigotimes_{k=1}^{e_{i}}(\pi_{i}\otimes\overline{\psi}_{k,i}).$$
	
	By definition of representations $\pi_{i}$, there is no integer $m$ such that $\pi_{i} \simeq \pi_{j} \otimes |\det|^{m}$(for any $i\neq j$) since all the $\alpha_{k,i}$ are distinct (for a fixed $i$) and $\alpha_{k,i} \alpha_{k',i'}^{-1} \neq q^{\pm 1} $(for any couples $(k',i')\neq (k,i)$). Then the segments $\Delta_{k,i} = \pi_{i}\otimes\overline{\psi}_{k,i}$ are not linked pairwise for any $k$ and $i$. Then it follows by Bernstein-Zelevisky classification \cite{MR584084}, that $i_{\overline{P}}^{G}(\rho \otimes \overline{\chi})$ is irreducible.
	
	We have just proved that if $\overline{\chi}$ is the unramified character of which corresponds to a maximal ideal $\mathfrak{M}$ of $\mathfrak{Z}_{D}$ above $\mathfrak{m} \in S$, then $i_{\overline{P}}^{G}(\rho \otimes \overline{\chi})$ is irreducible. By construction all the maximal ideals $\mathfrak{M}_{i}$ (which all lie above $\mathfrak{m} \in S$) are pairwise conjugated by some element $w \in W(D)$, so are the characters $\overline{\chi}_{i}$. Then for $\mathfrak{m} \in S$ all $i_{\overline{P}}^{G}(\rho \otimes \overline{\chi}_{i})$ are irreducible.
	
	Let $\mathfrak{m} \in S$, it follows from Frobenius reciprocity that $\mathrm{Hom}_{G}(i_{\overline{P}}^{G}(\rho \otimes \overline{\chi}_{i}),i_{\overline{P}}^{G}(\rho \otimes \overline{\chi}_{j})) \neq 0$, for all $1\leq i \leq |W(D)|$ and $1\leq j \leq |W(D)|$, because there is a $w_{i,j} \in W(D)$ such that $\overline{\chi}_{i}=\overline{\chi}_{j}^{w_{i,j}}$. Then for all $1\leq i \leq |W(D)|$, $1\leq j \leq |W(D)|$, $i_{\overline{P}}^{G}(\rho \otimes \overline{\chi}_{i}) \simeq i_{\overline{P}}^{G}(\rho \otimes \overline{\chi}_{j})$, because all these representations are irreducible on $S$. Write $P(\chi):=i_{\overline{P}}^{G}(\rho \otimes \overline{\chi}_{i})$, for some integer $i$.
	
	Then on open dense set $S$ we get :
	$$ \mathrm{c\text{--} Ind}_{J}^{G} \lambda \otimes_{\mathfrak{Z}_{\Omega}} \kappa(\mathfrak{m}) \simeq P(\chi)^{\oplus |W(D)|}.$$
\end{proof}

\section{Intertwining of representations}\label{H.4}

In this section we collect some useful lemmas. Recall that $\Omega:=[M,\rho]_{G}$, where $\rho$ is a supercuspidal representation of $M$,  $D:=[M,\rho]_{M}$, $\mathfrak{Z}_{\Omega}$ is the centre of category $\mathcal{R}^{\Omega}(G)$, and $\mathfrak{Z}_{D}$ is the centre of category $\mathcal{R}^{D}(M)$.

\begin{lemma} \label{lem0.8} Let $\sigma$ be an irreducible $K$-representation. With the notations of \textbf{Proposition \ref{lem0.6}}, we have:
	
	$$\mathrm{Hom}_{G}(\mathrm{c\text{--} Ind}_{K}^{G}\sigma, P(\chi))= \mathrm{Hom}_{G}(\mathrm{c\text{--} Ind}_{K}^{G}\sigma \otimes_{\mathfrak{Z}_{\Omega}} \kappa(\mathfrak{m}), P(\chi)),$$
	
	\noindent where $P(\chi):=i_{\overline{P}}^{G}(\rho \otimes \overline{\chi})$ is an irreducible parabolic induction of a supercuspidal representation $\rho \otimes \overline{\chi}$ of a Levi subgroup of $G$ and $\overline{\chi}$ some character corresponding to the algebra homomorphism $\mathrm{X} : \mathfrak{Z}_{D} \rightarrow E$, as in Lemma \ref{lem0.15}, such that $\mathfrak{M} = \mathrm{Ker}(\mathrm{X})$ is a maximal ideal of $\mathfrak{Z}_{D}$ above $\mathfrak{m}$.
\end{lemma}

\begin{proof} Observe that the maximal ideal $\mathfrak{m}=\mathrm{Ker}(\mathfrak{Z}_{\Omega} \xrightarrow{\chi} E)$  kills $P(\chi)$. The assertion follows.
\end{proof}

\begin{lemma} \label{lem0.9} Let $\sigma_{1}$ and $\sigma_{2}$ be two irreducible $K$-representations. We have the following isomorphisms:
	\begin{align}
	& \mathrm{Hom}_{G}(\mathrm{c\text{--} Ind}_{K}^{G}\sigma_{1}, \mathrm{c\text{--}Ind}_{K}^{G}\sigma_{2}) \otimes_{\mathfrak{Z}_{\Omega}} \kappa(\mathfrak{m}) \nonumber \\ 
	& \simeq \mathrm{Hom}_{G}(\mathrm{c\text{--} Ind}_{K}^{G}\sigma_{1}\otimes_{\mathfrak{Z}_{\Omega}} \kappa(\mathfrak{m}), \mathrm{c\text{--}Ind}_{K}^{G}\sigma_{2}\otimes_{\mathfrak{Z}_{\Omega}} \kappa(\mathfrak{m}))\nonumber \\ 
	&  \simeq \mathrm{Hom}_{G}(\mathrm{c\text{--} Ind}_{K}^{G}\sigma_{1}, \mathrm{c\text{--}Ind}_{K}^{G}\sigma_{2}\otimes_{\mathfrak{Z}_{\Omega}} \kappa(\mathfrak{m})).\nonumber
	\end{align}	
\end{lemma}

\begin{proof} Since $\cI_K^G \sigma$ is a projective and is also finitely generated as a $G$-representation, we have an isomorphism:	
	\[ \mathrm{Hom}_G(\cI_K^G \sigma, B) \otimes_{\mathfrak{Z}_{\Omega}} \kappa(\mathfrak{m}) \simeq  \mathrm{Hom}_G(\cI_K^G \sigma, B \otimes_{\mathfrak{Z}_{\Omega}} \kappa(\mathfrak{m})),\]
	\noindent for any $G$-representation $B$. Indeed if $B$ is a free $\mathfrak{Z}_{\Omega}$-module, then the assertion is clear. For general $B$, take a presentation of $B$ as $\mathfrak{Z}_{\Omega}$-module and use the previous case. The other isomorphism is proven the same way as in the lemma above.
\end{proof}


\section{Computation of multiplicities}\label{H.5}

Recall the decomposition (\ref{decomp2}) of $\mathrm{Ind}_J^K \lambda$. Then the multiplicity of an irreducible $K$-representation $\sigma$ appearing in $\mathrm{Ind}_J^K \lambda$ is $m_{\sigma}:=\dim_{E} \mathrm{Hom}_{K}(\mathrm{Ind}_{J}^{K} \lambda, \sigma)$. Recall that $\widehat{K}$ denotes the set of all isomorphism classes of irreducible representations of $K$. Define $I_{\sigma}:=\mathrm{c\text{--} Ind}_{K}^{G}\sigma$. Now we can deduce the following result from Proposition \ref{lem0.6}:

\begin{coro} \label{coro1}
	Let $\sigma \in \widehat{K}$, viewed as $K$-representation,  and $\chi : \mathfrak{Z}_{\Omega} \rightarrow E$ a algebra homomorphism corresponding to maximal ideal $\mathfrak{m}=\mathrm{Ker}(\mathfrak{Z}_{\Omega} \xrightarrow{\chi} E)$ of $\mathfrak{Z}_{\Omega}$. Then there is an integer $n_{\sigma}$ and a Zariski dense set $S$ in $\mathrm{Spec}(\mathfrak{Z}_{\Omega})$ such that:
	$$\mathrm{c\text{--} Ind}_{K}^{G}\sigma \otimes_{\mathfrak{Z}_{\Omega}} \kappa(\mathfrak{m}) \simeq P(\chi)^{\oplus n_{\sigma}}$$
	
	\noindent $\forall \mathfrak{m} \in S$,  where $P(\chi):=i_{\overline{P}}^{G}(\rho \otimes \overline{\chi})$ an irreducible parabolic induction of a supercuspidal representation of a Levi subgroup of $G$ and $\overline{\chi}$ some unramified character corresponding to the algebra homomorphism $\mathrm{X} : \mathfrak{Z}_{D} \rightarrow E$, as in Lemma \ref{lem0.15}, such that $\mathfrak{M} = \mathrm{Ker}(\mathrm{X})$ is a maximal ideal of $\mathfrak{Z}_{D}$ above $\mathfrak{m}$.
	
	Moreover we have the following relations of multiplicities :
	$$\sum\limits_{\sigma \in \widehat{K}} m_{\sigma}n_{\sigma} =|W(D)|$$
	$$\sum\limits_{\sigma \in \widehat{K}} m_{\sigma}^{2} =|W(D)|.$$
\end{coro}

\begin{proof} It follows from decomposition (\ref{decomp2}) and from Proposition \ref{lem0.6} that:
	$$\bigoplus_{\sigma \in \widehat{K}} (\mathrm{c\text{--} Ind}_{K}^{G}\sigma \otimes_{\mathfrak{Z}_{\Omega}} \kappa(\mathfrak{m}))^{\oplus m_{\sigma}} \simeq P(\chi)^{\oplus |W(D)|}.$$
	
	\noindent Then we also have
	$$\bigoplus_{\sigma \in \widehat{K}} (\mathfrak{M}_{\lambda}(\mathrm{c\text{--} Ind}_{K}^{G}\sigma \otimes_{\mathfrak{Z}_{\Omega}} \kappa(\mathfrak{m})))^{\oplus m_{\sigma}} \simeq \mathfrak{M}_{\lambda}(P(\chi))^{\oplus |W(D)|}.$$
	
	Observe that by Proposition \ref{lem0.6} the representation $P(\chi)$ is irreducible, in particular is indecomposable. The same observation holds in the category of $\mathcal{H}(G,\lambda)$-modules for $\mathfrak{M}_{\lambda}(P(\chi))$. Moreover the $\mathcal{H}(G,\lambda)$-module $\mathfrak{M}_{\lambda}(\mathrm{c\text{--} Ind}_{K}^{G}\sigma \otimes_{\mathfrak{Z}_{\Omega}} \kappa(\mathfrak{m}))$ is of finite length hence by $\S$2, n$^{\circ}$5, Theorem 2. a) \cite{MR3027127} it can be written as a direct sum of indecomposable modules $I_k(\sigma)$:
	$$\mathfrak{M}_{\lambda}(\mathrm{c\text{--} Ind}_{K}^{G}\sigma \otimes_{\mathfrak{Z}_{\Omega}} \kappa(\mathfrak{m})) = \bigoplus_k I_k(\sigma).$$
	
	\noindent Then, again, by theorem of by Krull-Remak-Schmidt Theorem($\S$2, n$^{\circ}$5, Theorem 2. b) \cite{MR3027127}), the decomposition:
	$$\bigoplus_{\sigma \in \widehat{K}} \bigoplus_k I_k(\sigma)^{\oplus m_{\sigma}}  \simeq \mathfrak{M}_{\lambda}(P(\chi))^{\oplus |W(D)|},$$
	
	\noindent into indecomposable sub-modules is unique up to permutation of factors.  This theorem is applicable because all the modules in the direct sum are of finite length. It follows that by the uniqueness of such a decomposition there exists an integer $n_{k,\sigma}$, such that:
	$$I_k(\sigma) \simeq \mathfrak{M}_{\lambda}(P(\chi))^{\oplus n_{k,\sigma}}$$
	
	\noindent Then there exists an integer $ n_{\sigma}:=\sum_k n_{k,\sigma}$ (that may depend on $\chi$ as well) such that:
	$$\mathfrak{M}_{\lambda}(\mathrm{c\text{--} Ind}_{K}^{G}\sigma \otimes_{\mathfrak{Z}_{\Omega}} \kappa(\mathfrak{m})) \simeq \mathfrak{M}_{\lambda}(P(\chi))^{\oplus n_{\sigma}}.$$
	
	\noindent So the same holds for representations:
	$$\mathrm{c\text{--} Ind}_{K}^{G}\sigma \otimes_{\mathfrak{Z}_{\Omega}} \kappa(\mathfrak{m}) \simeq P(\chi)^{\oplus n_{\sigma}}.$$
	
	\noindent Then by definition of $n_{\sigma}$ we have:
	$$\sum\limits_{\sigma \in \widehat{K}} m_{\sigma}n_{\sigma} =|W(D)|.$$
	
	Let's compute $\dim_{E} \mathrm{Hom}_{K}(\mathrm{Ind}_{J}^{K}\lambda, \mathrm{Ind}_{J}^{K}\lambda)$ in two different ways. By restriction induction formula we have 
	$$\mathrm{Res}_{J}^{K} \mathrm{Ind}_{J}^{K}\lambda = \bigoplus_{\overline{g} \in J\setminus K /J} \mathrm{Ind}_{K \cap J^{g}}^{K} \mathrm{Res}_{K \cap J^{g}}^{J^{g}} \lambda^{g},$$
	
	\noindent where $g$ is a representative of $\overline{g}$. Then combining it with Frobenius reciprocity we get:
	$$\mathrm{Hom}_{K}(\mathrm{Ind}_{J}^{K}\lambda, \mathrm{Ind}_{J}^{K}\lambda) = \bigoplus_{\overline{g} \in J\setminus K /J} \mathrm{Hom}_{J \cap J^{g}}(\lambda, \lambda^{g})$$
	
	\noindent By definition the space $\mathrm{Hom}_{J \cap J^{g}}(\lambda, \lambda^{g})$ is the intertwining space. Assume first that $(J, \lambda)$ is a simple type. In the course of this proof we will use the same notation from the book \cite{MR1204652}. Let $\Gamma=F[\beta]$ be a finite extension of $F$, which is denoted $E$ in the chapter 5 of that book. Then according to \cite{MR1204652} (5.5.11) $g \in G$ intertwines $\lambda$ if and only if $g \in J\tilde{W}(\mathfrak{B})J$. So $g \in K$ intertwines $\lambda$ if and only if $g \in J\tilde{W}(\mathfrak{B})J \cap K = J W_{0}(\mathfrak{B}) J$ and $|W_{0}(\mathfrak{B})| = e(\mathfrak{B}|\mathfrak{o}_{\Gamma})!$ by construction. In simple type case we have then $\dim_{E} \mathrm{Hom}_{K}(\mathrm{Ind}_{J}^{K}\lambda, \mathrm{Ind}_{J}^{K}\lambda) = e(\mathfrak{B}|\mathfrak{o}_{\Gamma})! = |W(D)|$.

	In general case we have to deal with semi-simple types. The reference is \cite{MR1711578}. Let $\overline{M}$ be a unique Levi subgroup of $G$ which contains the $N_{G}(M)$-stabilizer of the inertia class $D$ and is minimal for this property. The Levi subgroup $\overline{M}$ is the $G$-stabilizer of a decomposition $V=\bigoplus_{i=1}^{s} W_{i}$ of $V \simeq F^{n}$ as a sum of non-zero subspaces $W_{i}$. Set $G_{i}=\mathrm{Aut}_{F}(W_{i})$. We then have  $M=\prod_{i=1}^{s} M_{i}$ and $K \cap M=\prod_{i=1}^{s} K_{i}$, where $M_{i}=M \cap G_{i}=GL_{n_{i}}(F)^{e_{i}}$ and $K_{i}=K \cap G_{i}$. The type $(J_{M}, \lambda_{M})$ decomposes as a tensor product of types $(J_{M_{i}}, \lambda_{M_{i}})$, each of which admits a $G_{i}$-cover $(J_{\overline{M}_{i}}, \lambda_{\overline{M}_{i}})$ as in \cite{MR1711578} section 1.4. We put $J_{\overline{M}}=\prod_{i=1}^{s} J_{\overline{M}_{i}}$ and $\lambda_{\overline{M}}=\bigotimes_{i=1}^{s} \lambda_{\overline{M}_{i}}$. The main theorem asserts that $(J, \lambda)$ is a $G$-cover of $(J_{\overline{M}}, \lambda_{\overline{M}})$.
	
	It follows from corollary 1.6 in \cite{MR1711578}, that $g \in G$ intertwines $\lambda$ if and only if it is of the form $g=j_{1}mj_{2}$, where $j_{1}$ and $j_{2}$ are in $J$ and $m \in \overline{M}$, which intertwines $\lambda_{\overline{M}}$. The element $m$ can be written as $m=m_{1} \otimes\ldots\otimes m_{s}$, where $m_{i} \in M_{i}$ intertwine $\lambda_{\overline{M}_{i}}$.
	Then according to \cite{MR1204652} (5.5.11) $m_{i} \in M_{i}$ intertwine $\lambda_{\overline{M}_{i}}$  if and only if $m \in J_{\overline{M}_{i}}\tilde{W}(\mathfrak{B}_{i})J_{\overline{M}_{i}}$ (with analogous notations to 5.5\cite{MR1204652}). This shows that $m \in \overline{M}$ intertwine $\lambda_{\overline{M}}$ if and only if $m \in J_{\overline{M}}(\prod_{i=1}^{s} \tilde{W}(\mathfrak{B}_{i}))J_{\overline{M}}$.
	
	The decomposition
	$$\mathrm{Hom}_{K}(\mathrm{Ind}_{J}^{K}\lambda, \mathrm{Ind}_{J}^{K}\lambda) = \bigoplus_{\overline{g} \in J\setminus K /J} \mathrm{Hom}_{J \cap J^{g}}(\lambda, \lambda^{g}),$$
	
	\noindent where $g$ is a representative of $\overline{g}$,  shows that
	\begin{align}
	& \dim_{E} \mathrm{Hom}_{K}(\mathrm{Ind}_{J}^{K}\lambda, \mathrm{Ind}_{J}^{K}\lambda) = |J \setminus\left\lbrace g \in K | g \:\text{intertwines}\: \lambda \right\rbrace /J|\nonumber \\ 
	&=|J \setminus\left\lbrace g \in K | g=j_{1}mj_{2}, j_{1} \text{ and } j_{2} \text{ are in } J \text{ and } m \in \overline{M}  \text{ intertwines }  \lambda_{\overline{M}}  \right\rbrace /J| \nonumber \\ 
	&=|J \setminus\left\lbrace g \in K | g=j_{1}mj_{2}; j_{1},j_{2} \in  J \text{ and } m \in J_{\overline{M}}(\prod_{i=1}^{s} \tilde{W}(\mathfrak{B}_{i}))J_{\overline{M}}  \right\rbrace /J| \nonumber\\
	& =|J \setminus K \cap (JJ_{\overline{M}}(\prod_{i=1}^{s} \tilde{W}(\mathfrak{B}_{i}))J_{\overline{M}}J)/J|\nonumber \\ 
	& =|J \setminus (JJ_{\overline{M}}(\prod_{i=1}^{s} K_{i} \cap  \tilde{W}(\mathfrak{B}_{i}))J_{\overline{M}}J)/J|\nonumber
	\end{align} 
	\begin{flushleft}
		$=|\prod_{i=1}^{s} W_{0}(\mathfrak{B}_{i})|= \prod_{i=1}^{s}| W(D_{i})| =|W(D)|.$	
	\end{flushleft}

	\noindent Hence in every case
	$$\dim_{E} \mathrm{Hom}_{K}(\mathrm{Ind}_{J}^{K}\lambda, \mathrm{Ind}_{J}^{K}\lambda) = |W(D)|.$$
	
	We have the following decomposition $\mathrm{Ind}_{J}^{K} \lambda = \bigoplus_{\sigma \in \widehat{K}} \sigma^{\oplus m_{\sigma}}$, by definition of multiplicities. Then
	$$\dim_{E} \mathrm{Hom}_{K}(\mathrm{Ind}_{J}^{K}\lambda, \mathrm{Ind}_{J}^{K}\lambda) = \sum\limits_{\sigma \in \widehat{K}} m_{\sigma}^{2}.$$
\end{proof}

The rest of this section will be focused on proving that $m_{\sigma}=n_{\sigma}$. Let $\mathcal{R}_{\lambda}(K)$ denote the category of smooth $K$-representations generated by their $\lambda$-isotypic subspace.

\begin{lemma}\label{lem0.11}
	The category $\mathcal{R}_{\lambda}(K)$ is abelian.
\end{lemma}

\begin{proof} Observe that any smooth $K$-representation is semi-simple, since $K$ is compact. Then it is enough to show that a smooth representation of $K$ is an object of $\mathcal{R}_{\lambda}(K)$ if and only if all of its irreducible subquotients are generated by their $\lambda$-isotypic subspaces.

	An irreducible summand $\sigma$ of $\mathrm{Ind}_J^K \lambda$ will be generated by  its $\lambda$-isotypic subspace. Indeed 
	$$n=\dim_E\mathrm{Hom}_{K}( \sigma, \mathrm{Ind}_{J}^{K} \lambda)= \dim_E \mathrm{Hom}_{J}(\sigma, \lambda)= \dim_E \mathrm{Hom}_{J}(\lambda, \sigma)\neq 0,$$
	
	\noindent because the restriction of $\sigma$ to $J$ is semi-simple. Therefore the $K$-stable subspace of $\sigma$ generated by $\lambda^{\oplus n}$ is non zero. So is the whole $\sigma$, since this $K$-representation is irreducible. If $M$ is any module of $\mathcal{H}(K, J,\lambda):= \mathrm{End}_{K}(\mathrm{Ind}_{J}^{K} \lambda)$, then by writing $M$ as a quotient of free module, we deduce that $M\otimes_{\mathcal{H}(K, J,\lambda)} \mathrm{Ind}_{J}^{K} \lambda$ is  a quotient  of a direct sum of copies of $\mathrm{Ind}_{J}^{K} \lambda$.  Thus irreducible subquotients of  $M\otimes_{\mathcal{H}(K, J,\lambda)} \mathrm{Ind}_J^K \lambda$ will be subquotients of $\mathrm{Ind}_{J}^{K} \lambda$. But this means that all irreducible subquotients are generated by the $\lambda$-isotypic subspace.
	
	Let now $\delta$ any object of $\mathcal{R}_{\lambda}(K)$. Since $\lambda$ is a type, observe that $M(\delta):=\mathrm{Hom}_{J}(\lambda, (\mathrm{ c\text{--}Ind}_K^G \delta)|J)=\mathrm{Hom}_{K}(\mathrm{Ind}_J^K\lambda, (\mathrm{ c\text{--}Ind}_K^G \delta)|K)$ is a non zero $\mathcal{H}(G,\lambda)$-module, which can be viewed as $\mathcal{H}(K, J,\lambda)$-module via restriction. Moreover all the irreducible subquotients of $\delta$ appear in $M(\delta)\otimes_{\mathcal{H}(K, J,\lambda)} \mathrm{Ind}_J^K \lambda$. By the second paragraph it follows that all irreducible subquotients of $\delta$ are generated by the $\lambda$-isotypic subspace. The other direction is trivial.
\end{proof}

\begin{lemma}\label{lem0.12}
	The categories $\mathcal{R}_{\lambda}(K)$ and $\mathcal{H}(K, J,\lambda)-Mod$ are equivalent, where $\mathcal{H}(K, J,\lambda):= \mathrm{End}_{K}(\mathrm{Ind}_{J}^{K} \lambda)$.
\end{lemma}

\begin{proof} This is a consequence of a more general result \cite[Section 1 (4.1)]{MR0249491}. 
\end{proof}

\begin{lemma}\label{lem0.10}
	Let $\sigma \in \widehat{K}$, where $\widehat{K}$ is a set of all isomorphism classes of irreducible $K$-representations. Write $m_{\sigma}:=\dim_{E} \mathrm{Hom}_{K}(\sigma, \mathrm{Ind}_{J}^{K} \lambda)$ for its multiplicity. Then:
	$$m_{\sigma} = \dim_{E} \mathrm{Hom}_{G}(\mathrm{c\text{--} Ind}_{K}^{G}\sigma, P(\chi)) = n_{\sigma}.$$
\end{lemma}

\begin{proof}   We claim that $\mathrm{Hom}_{G}(\mathrm{c\text{--} Ind}_{J}^{G}\lambda, P(\chi))$  is a free rank one module over algebra $\mathcal{H}(K, J, \lambda) = \left\lbrace f \in \mathcal{H}(G, J, \lambda) | \mathrm{supp}(f) \subset K\right \rbrace$, where $\mathcal{H}(G, J, \lambda):=\mathrm{End}_G(\cI_J^G \lambda)$. Let's first deal with a particular case before dealing with the general case.
	
	\noindent \textbf{1. Simple type case. } Assume that $\lambda$ is a simple type. It follows from (5.6) of \cite{MR1204652} that there is a support preserving isomorphism of Hecke algebras $\mathcal{H}(G_{\Gamma}, I_{\Gamma}, 1)  \simeq \mathcal{H}(G, J, \lambda)$, where $\Gamma$ is a well defined extension of $F$(denoted $K$ in \cite{MR1204652}), $G_{\Gamma}=GL_{e}(\Gamma)$ with $I_{\Gamma}$ the Iwahori subgroup of $G_{\Gamma}$ and $K_{\Gamma}$ be a maximal compact subgroup of $G_{\Gamma}$. Let $M=\mathrm{Hom}_{G}(\mathrm{c\text{--} Ind}_{J}^{G}\lambda, P(\chi))=\mathrm{Hom}_{J}(\lambda, P(\chi)|J) =\mathrm{Hom}_{K}(\mathrm{Ind}_{J}^{K}\lambda, P(\chi)|K)$, this is an $\mathcal{H}(G, J, \lambda)$-module.
	
	Notice that when $P(\chi) = i_{\overline{P}}^{G} (\rho \otimes \overline{\chi})$ is irreducible, the $\mathcal{H}(G, J, \lambda)$-module $M$ is simple. The module $M$ is also naturally an $\mathcal{H}(G_{\Gamma}, I_{\Gamma}, 1) $-module, and corresponds to an irreducible representation $M \otimes_{\mathcal{H}(G_{\Gamma}, I_{\Gamma}, 1)} \mathrm{c\text{--} Ind}_{I_{\Gamma}}^{G_{\Gamma}} 1 \simeq i_{B_{\Gamma}}^{G_{\Gamma}} \overline{\chi}_{\Gamma}$, where $\overline{\chi}_{\Gamma}$ is an unramified character of Borel subgroup $B_{\Gamma}$ of $G_{\Gamma}$, making $i_{B_{\Gamma}}^{G_{\Gamma}} \overline{\chi}_{\Gamma}$ irreducible. Notice that $M=\mathrm{Hom}_{J}(\lambda, P(\chi)|J)$ does not depend on the character $\overline{\chi}$, so that discussion above is always valid.
	$$\mathrm{Hom}_{G_{\Gamma}}(\mathrm{c\text{--} Ind}_{I_{\Gamma}}^{G_{\Gamma}} 1, i_{B_{\Gamma}}^{G_{\Gamma}} \overline{\chi}_{\Gamma}) = \mathrm{Hom}_{G_{\Gamma}}(\mathrm{c\text{--} Ind}_{I_{\Gamma}}^{G_{\Gamma}} 1, M \otimes_{\mathcal{H}(G_{\Gamma}, I_{\Gamma}, 1)} \mathrm{c\text{--} Ind}_{I_{\Gamma}}^{G_{\Gamma}} 1)= M.$$
	
	According to the description of Hecke algebras in section (5.6) of \cite{MR1204652} the isomorphism of Hecke algebras $t:\mathcal{H}(G_{\Gamma}, I_{\Gamma}, 1)  \simeq \mathcal{H}(G, J, \lambda)$ is support preserving,  in the sense that $\mathrm{supp}(tf)=J.\mathrm{supp}(f).J$, we have also a natural isomorphism between $\mathcal{H}(K_{\Gamma}, I_{\Gamma}, 1)= \left\lbrace f \in \mathcal{H}(G_{\Gamma}, I_{\Gamma}, 1) | \mathrm{supp}(f) \subset K_{\Gamma}\right \rbrace$ and $\mathcal{H}(K, J, \lambda)= \left\lbrace f \in \mathcal{H}(G, J, \lambda) | \mathrm{supp}(f) \subset K\right \rbrace$. Then we have:
	\begin{align}
	&\mathrm{Hom}_{K}(\mathrm{Ind}_{J}^{K}\lambda, P(\chi)|K)  = \nonumber \\
	&=\mathrm{Hom}_{G}(\mathrm{c\text{--} Ind}_{J}^{G}\lambda, P(\chi))=\mathrm{Hom}_{G_{\Gamma}}(\mathrm{c\text{--} Ind}_{I_{\Gamma}}^{G_{\Gamma}} 1, i_{B_{\Gamma}}^{G_{\Gamma}} \overline{\chi}_{\Gamma}) \nonumber\\
	&=\mathrm{Hom}_{K_{\Gamma}}(\mathrm{c\text{--} Ind}_{I_{\Gamma}}^{K_{\Gamma}} 1, i_{B_{\Gamma} \cap K_{\Gamma}}^{K_{\Gamma}} 1) = \mathrm{Hom}_{K_{\Gamma}}(\mathrm{Ind}_{I_{\Gamma}}^{K_{\Gamma}} 1, \mathrm{Ind}_{B_{\Gamma} \cap K_{\Gamma}}^{K_{\Gamma}} 1)  \nonumber\\ &=\mathrm{Hom}_{K_{\Gamma}}(\mathrm{Ind}_{I_{\Gamma}}^{K_{\Gamma}} 1, (\mathrm{Ind}_{B_{\Gamma} \cap K_{\Gamma}}^{K_{\Gamma}} 1)^{K_{1}}) = \mathrm{Hom}_{K_{\Gamma}}(\mathrm{Ind}_{I_{\Gamma}}^{K_{\Gamma}} 1, \mathrm{Ind}_{I_{\Gamma}}^{K_{\Gamma}} 1) \nonumber\\
	&= \mathcal{H}(K_{\Gamma}, I_{\Gamma}, 1) = \mathcal{H}(K, J, \lambda), \nonumber
	\end{align}

	\noindent where $K_{1} = \left\lbrace g \in G_{\Gamma} | g \equiv 1 \mod \mathfrak{p}_{\Gamma} \right\rbrace $, where $\mathfrak{p}_{\Gamma}$ is the maximal ideal in the ring of integers of $\Gamma$.
	
	\noindent \textbf{2. General case. }Let now $\lambda$ be some general semi-simple type. The second part of Main Theorem of section 8 in \cite{MR1711578} gives a support preserving Hecke algebra isomorphism $j :\mathcal{H}(\overline{M}, \lambda_{M}) \rightarrow \mathcal{H}(G, \lambda)$,where $\overline{M}$ is a unique Levi subgroup of $G$ which contains the $N_{G}(M)$-stabilizer of the inertia class $D$ and is minimal for this property. Moreover the section 1.5 \cite{MR1711578} gives a tensor product decomposition $\mathcal{H}(\overline{M}, \lambda_{M}) = \mathcal{H}_{1} \otimes_{E}\ldots\otimes_{E}\mathcal{H}_{s}$, where $\mathcal{H}_{i} = \mathcal{H}(G_{i}, J_{i}, \lambda_{i})$ is an affine Hecke algebras of type A and $(J_{i}, \lambda_{i})$ is some simple type with $G_{i}$ some general linear group over a $p$-adic field. 
	
	Let now $M = \mathrm{Hom}_{K}(\mathrm{Ind}_{J}^{K}\lambda, P(\chi)|K)$. The same argument as in the simple type case shows that $M$ is a simple $\mathcal{H}(G, \lambda)$-module. Then by \cite{MR3027127} \S 12 Proposition 2, $M$ is a quotient of $M_{1} \otimes_{E}\ldots\otimes_{E}M_{s}$, where $M_{i}$ is a simple $\mathcal{H}_{i}$-module. Since $E$ is algebraically closed we have by \cite{MR3027127} \S 12 Theorem 1 part a), that $M_{1} \otimes_{E}\ldots\otimes_{E}M_{s}$ is a simple $\mathcal{H}_{1} \otimes_{E}\ldots\otimes_{E}\mathcal{H}_{s}$-module. Then it follows that $M \simeq M_{1} \otimes_{E}\ldots\otimes_{E}M_{s}$.
	
	Let $K_{i}$ denote the maximal compact subgroup of $G_{i}$ and $\mathcal{A}_{i} :=\mathcal{H}(K_{i}, J_{i}, \lambda_{i})$. Then by simple type case $M_{i} \simeq \mathcal{A}_{i}$.

	Since the isomorphism $j$ above is support preserving, we have a similar decomposition for $\mathcal{H}(K, J, \lambda)$, namely $\mathcal{H}(K, J, \lambda) \simeq  \mathcal{A}_{1} \otimes_{E}\ldots\otimes_{E}\mathcal{A}_{s}$.  Then $M \simeq M_{1} \otimes_{E}\ldots\otimes_{E}M_{s} \simeq \mathcal{A}_{1} \otimes_{E}\ldots\otimes_{E}\mathcal{A}_{s} \simeq \mathcal{H}(K, J, \lambda)$.
	
	\medskip
	
	By the discussion above we always have $\mathrm{Hom}_{K}(\mathrm{Ind}_{J}^{K}\lambda, P(\chi)|K)=$ \\ \noindent $\mathcal{H}(K, J, \lambda)$. Let now $\tau$ be a subrepresentation of $P(\chi)|K$, generated by the $\lambda$-isotypic subspace of $P(\chi)|K$. Then by the isomorphisms above we have
	$$\mathrm{Hom}_{K}(\mathrm{Ind}_{J}^{K}\lambda, \tau)=\mathrm{Hom}_{K}(\mathrm{Ind}_{J}^{K}\lambda, P(\chi)|K)=\mathcal{H}(K, J, \lambda),$$

	\noindent as isomorphisms of $\mathcal{H}(K, J, \lambda)$-modules. Then by equivalence of categories (as in Lemma \ref{lem0.10}):
	$$\mathrm{Ind}_{J}^{K} \lambda \simeq \tau \hookrightarrow P(\chi)| K,$$
	
	\noindent where the first arrow from the left comes from the discussion above and the second one is a natural inclusion. Applying $\mathrm{Hom}_{K}(\sigma, .)$ to previous injection and then taking the dimensions of both sides yields an inequality:
	$$m_{\sigma} \leq \dim_{E} \mathrm{Hom}_{G}(\mathrm{c\text{--} Ind}_{K}^{G}\sigma, P(\chi)).$$
	
	Moreover by Lemma \ref{lem0.8} we have: 
	$$\dim_{E} \mathrm{Hom}_{G}(I_{\sigma}, P(\chi)) = \dim_{E} \mathrm{Hom}_{G}(I_{\sigma} \otimes_{\mathfrak{Z}_{\Omega}} \kappa(\mathfrak{m}), P(\chi)). $$
	Since by Corollary \ref{coro1} we have $I_{\sigma} \otimes_{\mathfrak{Z}_{\Omega}} \kappa(\mathfrak{m}) \simeq P(\chi)^{\oplus n_{\sigma}}$, then
	
	\begin{align}
	&\dim_{E} \mathrm{Hom}_{G}(I_{\sigma}, P(\chi)) = \dim_{E} \mathrm{Hom}_{G}(I_{\sigma} \otimes_{\mathfrak{Z}_{\Omega}} \kappa(\mathfrak{m}), P(\chi)) \nonumber \\ 
	&= \dim_{E} \mathrm{Hom}_{G}(P(\chi)^{\oplus n_{\sigma}}, P(\chi)) = n_{\sigma} \nonumber
	\end{align}
	
	The inequality $m_{\sigma} \leq \dim_{E} \mathrm{Hom}_{G}(\mathrm{c\text{--} Ind}_{K}^{G}\sigma, P(\chi)) = n_{\sigma}$ is actually an equality because of the relations
	$$\sum\limits_{\sigma \in \widehat{K}} m_{\sigma}n_{\sigma} =|W(D)|$$
	$$\sum\limits_{\sigma \in \widehat{K}} m_{\sigma}^{2} =|W(D)|,$$
	
	\noindent which were proven in Corollary \ref{coro1}.
\end{proof}

\section{Consequences}\label{H.6}

In this section we will deduce a few results from previous sections and finally prove that if $m_{\sigma}=1$ then :$$\mathfrak{Z}_{\Omega} \simeq \mathrm{End}_{G}(\mathrm{c\text{--} Ind}_{K}^{G} \sigma).$$
\noindent Let us denote $I_{\sigma}:=\mathrm{c\text{--} Ind}_{K}^{G}\sigma$, $A(\sigma_{1},\sigma_{2}):= \mathrm{Hom}_{G}(\mathrm{c\text{--} Ind}_{K}^{G}\sigma_{1}, \mathrm{c\text{--}Ind}_{K}^{G}\sigma_{2})$ and $A_{\sigma}:=A(\sigma,\sigma)$. We begin with the following result:

\begin{thm} \label{prop1}
	Let $\sigma \in \widehat{K}$, where $\widehat{K}$ is a set of all isomorphism classes of irreducible $K$-representations. Write $m_{\sigma}:=\dim_{E} \mathrm{Hom}_{K}(\mathrm{c\text{--} Ind}_{J}^{K} \lambda, \sigma)$  for its multiplicity. Then $\mathrm{Hom}_G(\cI_K^G \sigma_1, \cI_K^G \sigma_2)$ is a  free $\mathfrak{Z}_{\Omega}$-module of rank $m_{\sigma_1}m_{\sigma_2}$, for all $\sigma_{1}, \sigma_{2} \in \widehat{K}$ .
\end{thm}

\begin{proof} Write $A(\sigma_{1},\sigma_{2})$ for $\mathrm{Hom}_{G}(\mathrm{c\text{--} Ind}_{K}^{G}\sigma_{1}, \mathrm{c\text{--}Ind}_{K}^{G}\sigma_{2})$. Recall the decomposition (\ref{eq:1}) :
	$$\mathrm{c\text{--} Ind}_{J}^{G} \lambda = \bigoplus_{\sigma \in \widehat{K}} (\mathrm{c\text{--} Ind}_{K}^{G}\sigma)^{\oplus m_{\sigma}},$$
	
	\noindent so that: 
	$$\mathcal{H}(G,\lambda) \simeq \mathrm{End}_{G}(\mathrm{c\text{--} Ind}_{J}^{G} \lambda) = \prod_{\sigma_{1},\sigma_{2} \in \widehat{K}} \mathrm{Hom}_{G}(\mathrm{c\text{--} Ind}_{K}^{G}\sigma_{1}, \mathrm{c\text{--} Ind}_{K}^{G}\sigma_{2})^{m_{\sigma_{1}}m_{\sigma_{2}}}.$$
	
	Moreover the action of $\mathfrak{Z}_{\Omega}$ on $A(\sigma_{1},\sigma_{2})$ by multiplication, makes $A(\sigma_{1},\sigma_{2})$ into a sub-$\mathfrak{Z}_{\Omega}$-module of $\mathcal{H}(G,\lambda)$, via the previous decomposition. Let's prove that $A(\sigma_{1},\sigma_{2})$ is also a locally free finitely generated $\mathfrak{Z}_{\Omega}$-module.
	
	By the decomposition of $\mathcal{H}(G,\lambda)$, $A(\sigma_{1},\sigma_{2})$ is also a direct summand of $\mathcal{H}(G,\lambda)$.  Moreover by Lemma \ref{lem0.3}, $\mathcal{H}(G,\lambda)$ is a free $\mathfrak{Z}_{\Omega}$-module, and therefore $A(\sigma_{1},\sigma_{2})$ is a projective $\mathfrak{Z}_{\Omega}$-module. Let $\mathfrak{m}$ be a maximal ideal of $\mathfrak{Z}_{\Omega}$ and let  $d_{\mathfrak{m}}$ be the rank of $(A(\sigma_{1},\sigma_{2}))_{\mathfrak{m}}$. Then:
	\begin{align}
	d_{\mathfrak{m}} & = \dim_{E} (A(\sigma_{1},\sigma_{2})_{\mathfrak{m}} \otimes_{(\mathfrak{Z}_{\Omega})_{\mathfrak{m}}} \kappa(\mathfrak{m})) = \dim_{E} (A(\sigma_{1},\sigma_{2})\otimes_{\mathfrak{Z}_{\Omega}} (\mathfrak{Z}_{\Omega})_{\mathfrak{m}} \otimes_{(\mathfrak{Z}_{\Omega})_{\mathfrak{m}}} \kappa(\mathfrak{m}))\nonumber \\ 
	& = \dim_{E} A(\sigma_{1},\sigma_{2})\otimes_{\mathfrak{Z}_{\Omega}} \kappa(\mathfrak{m})\nonumber   
	\end{align}
	\noindent (recall that $\kappa(\mathfrak{m}) \simeq E$). We will prove now that the local rank is constant on a dense set of maximal ideals.
	
	Let $i \in \{1,2\}$. Choose now  $\mathfrak{m}=\mathrm{Ker}(\mathfrak{Z}_{\Omega} \xrightarrow{\chi} E) \in S$ (see Proposition \ref{lem0.6} for definition of the set $S$). By Corollary \ref{coro1} there is an integer $n_{\sigma_{i}}$ such that:
	$$I_{\sigma_{i}} \otimes_{\mathfrak{Z}_{\Omega}} \kappa(\mathfrak{m}) \simeq P(\chi)^{\oplus n_{\sigma_{i}}}.$$
	
	\noindent Then
	\begin{align}
	& \dim_{E} \mathrm{Hom}_{G}(I_{\sigma_{i}}, P(\chi)) = \dim_{E} \mathrm{Hom}_{G}(I_{\sigma_{i}} \otimes_{\mathfrak{Z}_{\Omega}} \kappa(\mathfrak{m}), P(\chi)) \nonumber \\ 
	& = \dim_{E} \mathrm{Hom}_{G}(P(\chi)^{\oplus n_{\sigma_{i}}}, P(\chi)) = n_{\sigma_{i}}.\nonumber 
	\end{align}
	
	\noindent By Lemma \ref{lem0.10} we have:
	$$m_{\sigma_{i}} = \dim_{E} \mathrm{Hom}_{G}(I_{\sigma_{i}}, P(\chi))= n_{\sigma_{i}}.$$
	
	\noindent Then
	\begin{align}
	& \mathrm{Hom}_{G}(I_{\sigma_{1}} \otimes_{\mathfrak{Z}_{\Omega}} \kappa(\mathfrak{m}), I_{\sigma_{2}} \otimes_{\mathfrak{Z}_{\Omega}} \kappa(\mathfrak{m})) \simeq \mathrm{Hom}_{G}( P(\chi)^{\oplus m_{\sigma_{1}}}, P(\chi)^{\oplus m_{\sigma_{2}}})\nonumber \\ 
	&  \simeq \mathrm{End}_{G}(P(\chi))^{ m_{\sigma_{1}} m_{\sigma_{2}}} \simeq (E)^{ m_{\sigma_{1}} m_{\sigma_{2}}},\nonumber 
	\end{align}
	
	\noindent since by Schur's lemma $\mathrm{End}_{G}(P(\chi)) = E$. Finally by Lemma \ref{lem0.9}
	\begin{align}
	&d_{\mathfrak{m}}= \dim_{E} A(\sigma_{1},\sigma_{2}) \otimes_{\mathfrak{Z}_{\Omega}} \kappa(\mathfrak{m}) = \dim_{E} \mathrm{Hom}_{G}(I_{\sigma_{1}} \otimes_{\mathfrak{Z}_{\Omega}} \kappa(\mathfrak{m}), I_{\sigma_{2}} \otimes_{\mathfrak{Z}_{\Omega}} \kappa(\mathfrak{m})) \nonumber \\ 
	& = m_{\sigma_{1}} m_{\sigma_{2}}. \nonumber 
	\end{align}
	
	This proves that $d_{\mathfrak{m}}= m_{\sigma_{1}} m_{\sigma_{2}}$, $\forall \mathfrak{m} \in S$. Moreover this equality is true for all $\mathfrak{m}$ since $\mathfrak{m} \mapsto d_{\mathfrak{m}}$ is locally constant function(\cite{MR782296} Chapitre 2 \S 5.2 Th\'{e}or\`{e}me 1 c)) and $S$ is a dense set. The main result of \cite{MR0469906} allows us to conclude.
\end{proof}

Now we deduce the result announced in the introduction of this paper:

\begin{coro} \label{coro2}
	Let $\sigma \in \widehat{K}$, such that $m_{\sigma}=1$. Then the canonical map $\mathfrak{Z}_{\Omega} \rightarrow \mathrm{End}_{G}(\mathrm{c\text{--} Ind}_{K}^{G} \sigma)$ induces a ring isomorphism:
	$$\mathfrak{Z}_{\Omega} \simeq \mathrm{End}_{G}(\mathrm{c\text{--} Ind}_{K}^{G} \sigma).$$
\end{coro}

\begin{rem}\label{mult} In the Iwahori case $1$ and $st$ are the only multiplicity free direct summands of $\mathrm{Ind}_I^K 1$. When the type $(J, \lambda)$ is simple $\sigma_{max}(\lambda)$ and $\sigma_{min}(\lambda)$  are  also the only multiplicity free summands of $\mathrm{Ind}_{J}^{K} \lambda$, because the Hecke algebra of a simple is isomorphic to Iwahori Hecke algebra (\cite[(5.6)]{MR1204652}). Assume now that the type $(J, \lambda)$ is semi-simple $[M,\rho]_{G}$-type and is a cover of $(J_1 \times\ldots \times J_r, \lambda_1 \otimes\ldots \otimes \lambda_r)$, where the types $(J_i, \lambda_i)$ are all simple and distinct. Let $P=MN$ be a parabolic subgroup of $G$. Then according to the end of section 6 in \cite{MR1728541}, the restriction of the $K$-representation $\sigma:=\sigma_{\mathcal{P}}(\lambda)$ to $K \cap N$ is trivial, and $\sigma|K\cap M \simeq \sigma_{1}\otimes\ldots\otimes \sigma_{r}$ where $\sigma_{i} := \sigma_{\mathcal{P}_{i}}(\lambda_{i})$. It follows that $m_{\sigma}:=\dim_{E} \mathrm{Hom}_{K}(\mathrm{Ind}_{J}^{K} \lambda, \sigma)=1$ if and only if for each $i$, $\sigma_{i} = \sigma_{max}(\lambda_{i})$ or $\sigma_{i} = \sigma_{min}(\lambda_{i})$. In this case there will be $2^r$ multiplicity free direct summands of $\mathrm{Ind}_{J}^{K} \lambda$.
\end{rem}

\subsection*{Acknowledgments}   The results of this paper are the main part of the author's PhD thesis. The author tremendously grateful to his advisor Vytautas Pa\v{s}k\={u}nas for sharing his ideas with the author and for many helpful discussions. The author would also like to thank the referee for useful comments and corrections. This work was supported by SFB/TR 45 of the DFG.

\bibliographystyle{alpha}
\addcontentsline{toc}{section}{References}
\bibliography{bib}

\newcommand{\etalchar}[1]{$^{#1}$}
\begin{thebibliography}{CEG{\etalchar{+}}16}

\bibitem[Bas68]{MR0249491}
Hyman Bass.
\newblock {\em Algebraic {$K$}-theory}.
\newblock W. A. Benjamin, Inc., New York-Amsterdam, 1968.

\bibitem[Ber84]{MR771671}
J.~N. Bernstein.
\newblock Le ``centre'' de {B}ernstein.
\newblock In {\em Representations of reductive groups over a local field},
  Travaux en Cours, pages 1--32. Hermann, Paris, 1984.
\newblock Edited by P. Deligne.

\bibitem[BH06]{MR2234120}
Colin~J. Bushnell and Guy Henniart.
\newblock {\em The local {L}anglands conjecture for {$\rm GL(2)$}}, volume 335
  of {\em Grundlehren der Mathematischen Wissenschaften [Fundamental Principles
  of Mathematical Sciences]}.
\newblock Springer-Verlag, Berlin, 2006.

\bibitem[BK93]{MR1204652}
Colin~J. Bushnell and Philip~C. Kutzko.
\newblock {\em The admissible dual of {${\rm GL}(N)$} via compact open
  subgroups}, volume 129 of {\em Annals of Mathematics Studies}.
\newblock Princeton University Press, Princeton, NJ, 1993.

\bibitem[BK98]{MR1643417}
Colin~J. Bushnell and Philip~C. Kutzko.
\newblock Smooth representations of reductive {$p$}-adic groups: structure
  theory via types.
\newblock {\em Proc. London Math. Soc. (3)}, 77(3):582--634, 1998.

\bibitem[BK99]{MR1711578}
Colin~J. Bushnell and Philip~C. Kutzko.
\newblock Semisimple types in {${\rm GL}_n$}.
\newblock {\em Compositio Math.}, 119(1):53--97, 1999.

\bibitem[Bor76]{MR0444849}
Armand Borel.
\newblock Admissible representations of a semi-simple group over a local field
  with vectors fixed under an {I}wahori subgroup.
\newblock {\em Invent. Math.}, 35:233--259, 1976.

\bibitem[Bou85a]{MR782297}
Nicolas Bourbaki.
\newblock {\em \'El\'ements de math\'ematique}.
\newblock Masson, Paris, 1985.
\newblock Alg\`ebre commutative. Chapitres 5 \`a 7. [Commutative algebra.
  Chapters 5--7], Reprint.

\bibitem[Bou85b]{MR782296}
Nicolas Bourbaki.
\newblock {\em \'El\'ements de math\'ematique}.
\newblock Masson, Paris, 1985.
\newblock Alg\`ebre commutative. Chapitres 1 \`a 4. [Commutative algebra.
  Chapters 1--4], Reprint.

\bibitem[Bou03]{MR1994218}
Nicolas Bourbaki.
\newblock {\em Algebra {II}. {C}hapters 4--7}.
\newblock Elements of Mathematics (Berlin). Springer-Verlag, Berlin, 2003.
\newblock Translated from the 1981 French edition by P. M. Cohn and J. Howie,
  Reprint of the 1990 English edition [Springer, Berlin; MR1080964
  (91h:00003)].

\bibitem[Bou12]{MR3027127}
N.~Bourbaki.
\newblock {\em \'El\'ements de math\'ematique. {A}lg\`ebre. {C}hapitre 8.
  {M}odules et anneaux semi-simples}.
\newblock Springer, Berlin, 2012.
\newblock Second revised edition of the 1958 edition [MR0098114].

\bibitem[BZ77]{MR0579172}
I.~N. Bernstein and A.~V. Zelevinsky.
\newblock Induced representations of reductive {${\mathfrak{p}}$}-adic groups.
  {I}.
\newblock {\em Ann. Sci. \'Ecole Norm. Sup. (4)}, 10(4):441--472, 1977.

\bibitem[Cas]{Cass}
Bill Casselman.
\newblock {I}ntroduction to the {T}heory of {A}dmissible {R}epresentations of
  {$p$}-adic reductive groups.
\newblock http://www.math.ubc.ca/~cass/research/pdf/p-adic-book.pdf.

\bibitem[CEG{\etalchar{+}}16]{MR3529394}
Ana Caraiani, Matthew Emerton, Toby Gee, David Geraghty, Vytautas
  Pa{\v{s}}k\=unas, and Sug~Woo Shin.
\newblock Patching and the {$p$}-adic local {L}anglands correspondence.
\newblock {\em Camb. J. Math.}, 4(2):197--287, 2016.

\bibitem[Dat99a]{MR1676870}
J.-F. Dat.
\newblock Caract\`eres \`a valeurs dans le centre de {B}ernstein.
\newblock {\em J. Reine Angew. Math.}, 508:61--83, 1999.

\bibitem[Dat99b]{MR1670599}
J.-F. Dat.
\newblock Types et inductions pour les repr\'esentations modulaires des groupes
  {$p$}-adiques.
\newblock {\em Ann. Sci. \'Ecole Norm. Sup. (4)}, 32(1):1--38, 1999.
\newblock With an appendix by Marie-France Vign\'eras.

\bibitem[Hel16]{MR3508741}
David Helm.
\newblock The {B}ernstein center of the category of smooth {$W(k)[{\rm
  GL}_n(F)]$}-modules.
\newblock {\em Forum Math. Sigma}, 4:e11, 98, 2016.

\bibitem[How85]{MR821216}
Roger Howe.
\newblock {\em Harish-{C}handra homomorphisms for {$\mathfrak{p}$}-adic
  groups}, volume~59 of {\em CBMS Regional Conference Series in Mathematics}.
\newblock Published for the Conference Board of the Mathematical Sciences,
  Washington, DC; by the American Mathematical Society, Providence, RI, 1985.
\newblock With the collaboration of Allen Moy.

\bibitem[How94]{MR1303601}
Roger Howe.
\newblock Hecke algebras and {$p$}-adic {${\rm GL}_n$}.
\newblock In {\em Representation theory and analysis on homogeneous spaces
  ({N}ew {B}runswick, {NJ}, 1993)}, volume 177 of {\em Contemp. Math.}, pages
  65--100. Amer. Math. Soc., Providence, RI, 1994.

\bibitem[Jan95]{MR1341660}
Chris Jantzen.
\newblock On the {I}wahori-{M}atsumoto involution and applications.
\newblock {\em Ann. Sci. \'{E}cole Norm. Sup. (4)}, 28(5):527--547, 1995.

\bibitem[Kud94]{MR1265559}
Stephen~S. Kudla.
\newblock The local {L}anglands correspondence: the non-{A}rchimedean case.
\newblock In {\em Motives ({S}eattle, {WA}, 1991)}, volume~55 of {\em Proc.
  Sympos. Pure Math.}, pages 365--391. Amer. Math. Soc., Providence, RI, 1994.

\bibitem[Mat89]{MR1011461}
Hideyuki Matsumura.
\newblock {\em Commutative ring theory}, volume~8 of {\em Cambridge Studies in
  Advanced Mathematics}.
\newblock Cambridge University Press, Cambridge, second edition, 1989.
\newblock Translated from the Japanese by M. Reid.

\bibitem[Pyv18]{Pyv1}
Alexandre Pyvovarov.
\newblock On the {B}reuil-{S}chneider conjecture: {G}eneric case.
\newblock {\em Preprint, arxiv.org/abs/1803.01610}, 2018.

\bibitem[Ren10]{MR2567785}
David Renard.
\newblock {\em Repr\'esentations des groupes r\'eductifs {$p$}-adiques},
  volume~17 of {\em Cours Sp\'ecialis\'es [Specialized Courses]}.
\newblock Soci\'et\'e Math\'ematique de France, Paris, 2010.

\bibitem[Swa78]{MR0469906}
Richard~G. Swan.
\newblock Projective modules over {L}aurent polynomial rings.
\newblock {\em Trans. Amer. Math. Soc.}, 237:111--120, 1978.

\bibitem[SZ99]{MR1728541}
P.~Schneider and E.-W. Zink.
\newblock {$K$}-types for the tempered components of a {$p$}-adic general
  linear group.
\newblock {\em J. Reine Angew. Math.}, 517:161--208, 1999.
\newblock With an appendix by Schneider and U. Stuhler.

\bibitem[Zel80]{MR584084}
A.~V. Zelevinsky.
\newblock Induced representations of reductive {${\mathfrak{p}}$}-adic groups.
  {II}. {O}n irreducible representations of {${\rm GL}(n)$}.
\newblock {\em Ann. Sci. \'Ecole Norm. Sup. (4)}, 13(2):165--210, 1980.

\end{thebibliography}
\nocite{*}

\noindent Morningside Center of Mathematics, No.55 Zhongguancun Donglu, Academy of Mathematics and Systems Science, Beijing , Haidian District, 100190 China
\\
\textit{E-mail address}: pyvovarov@amss.ac.cn
\end{document}